\documentclass[preprint,aoas]{imsart}

\RequirePackage{amsthm,amsmath,amsfonts,amssymb}
\RequirePackage[authoryear]{natbib}
\RequirePackage[colorlinks,citecolor=blue,urlcolor=blue]{hyperref}
\RequirePackage{graphicx}

\usepackage{bm}
\usepackage{amsfonts}
\usepackage{subcaption}
\usepackage{multirow}

\startlocaldefs
\theoremstyle{plain}

\newtheorem{theorem}{Theorem}[section]
\newtheorem{lemma}[theorem]{Lemma}

\theoremstyle{remark}
\newtheorem{definition}[theorem]{Definition}

\endlocaldefs

\begin{document}

\begin{frontmatter}
\title{Fused $\ell_{1}$ Trend Filtering on Graphs}
\runtitle{Fused  $\ell_{1}$ Trend Filtering on Graphs}

\begin{aug}
\author{\fnms{Vladimir}~\snm{Pastukhov}}
\end{aug}

\begin{abstract}
This paper is dedicated to the fused trend filtering on a general graph, which is a combination of fused estimator and 1-st order trend filtering on a graph. There are two cases of fusion regularisers studied in this work: anisotropic total variation (i.e. fused lasso) and nearly-isotonic restriction. For the trend filtering part we consider general trend filtering on a given graph and Kronecker trend filter for the case of lattice data. We show how these estimators are related to each other and propose a computationally feasible numerical solution with a linear complexity per iteration with respect to the amount of edges in the graph.
\end{abstract}

\begin{keyword}
\kwd{Fused lasso}
\kwd{nearly-isotonic regression}
\kwd{trend filtering}
\kwd{graph smoothing}
\kwd{nonparametric regression}
\end{keyword}

\end{frontmatter}

\section{Introduction}
This work is motivated by recent papers in regularized nonparametric estimation of signals over multidimensional grids and over the nodes of general graphs. 
We study combination of fused lasso and nearly-isotonically constrained estimators with two versions of trend filters.

Let us consider the following nonparametric model 
\begin{equation*}\label{}
\bm{Y} = \bm{\mathring{\beta}} + \bm{\varepsilon},
\end{equation*}
where $\bm{\mathring{\beta}}\in\mathbb{R}^{n}$ is an unknown signal, and the error term $\bm{\varepsilon}\in\mathcal{N}(\bm{0}, \sigma^{2}\bm{I})$, with some $\sigma < \infty$. We start with one-dmencional case and introduce the estimators that are used in this paper. Throughout the paper, the notations "regression", "estimator", "approximator" and "filter" are used interchangeably. 

For a given sequence of data points $\bm{y} \in \mathbb{R}^{n}$ and some penalization parameter $\lambda_{F} > 0$ the fused lasso estimator (cf. \cite{hoefling2010path, rinaldo2009properties, tibshirani2011solution, tibshirani2005sparsity}) is given by
\begin{equation}\label{FL1d}
\hat{\bm{\beta}}^{FL}(\bm{y}, \lambda_{FL}) = \underset{\bm{\beta} \in \mathbb{R}^{n}}{\arg \min} \, \frac{1}{2} ||\bm{y} - \bm{\beta}||_{2}^{2} +\lambda_{F} \sum_{i=1}^{n-1}|\beta_{i} - \beta_{i+1}|.
\end{equation}
Fused lasso estimator transforms original signal $\bm{y}$ to the piece-wise constant sequence.

Next, in the shape-constrained inference the simplest constrained signal estimation problem is the isotonic estimator in one dimension. For a given sequence of data points $\bm{y} \in \mathbb{R}^{n}$ the isotonic regression is 
\begin{equation}\label{orMLEp}
\hat{\bm{\beta}}^{I} = \underset{\bm{\beta} \in \mathbb{R}^{n}}{\arg\min}||\bm{y} - \bm{\beta}||_{2}^{2}, \quad \text{subject to} \quad \beta_{1} \leq \beta_{2} \leq \dots \leq \beta_{n}, 
\end{equation}
i.e. it is the $\ell^{2}$-projection of the vector $\bm{y}$ onto the cone of non-increasing vectors in $\mathbb{R}^{n}$. For the  general introduction to the isotonic regression we refer to the monographs \cite{brunk1972statistical, robertson1988order}.

The less restrictive version of the isotonic regression is nearly-isotonic regression, which was introduced in \cite{tibshirani2011nearly} and studied for the case of general dimensions in \cite{minami2020estimating, pastukhov2022fused}. The nearly-isotonic regression is given by the following optimization problem
\begin{equation}\label{NI1d}
\hat{\bm{\beta}}^{NI}(\bm{y}, \lambda_{NI}) = \underset{\bm{\beta} \in \mathbb{R}^{n}}{\arg \min} \, \frac{1}{2} ||\bm{y} - \bm{\beta}||_{2}^{2} + \lambda_{NI}\sum_{i=1}^{n-1}|\beta_{i} - \beta_{i+1}|_{+},
\end{equation}
where $x_{+} =  x \cdot 1\{x > 0 \}$ and $\lambda_{NI} > 0$. The nearly-isotonic approximator makes the original input piece-wise monotone and for a large value of $\lambda_{NI}$ we have $\hat{\bm{\beta}}^{NI}(\bm{y}, \lambda_{NI}) = \hat{\bm{\beta}}^{I} $, i.e. it is equivalent to the isotonic regression.

Next, note that the estimators defined in (\ref{FL1d}) and (\ref{NI1d}) can be re-writen as
\begin{equation}\label{FL1d_D}
\hat{\bm{\beta}}^{FL}(\bm{y}, \lambda_{F}) = \underset{\bm{\beta} \in \mathbb{R}^{n}}{\arg \min} \, \frac{1}{2} ||\bm{y} - \bm{\beta}||_{2}^{2} +\lambda_{F} \sum_{i=1}^{n-1}|D^{f}\bm{\beta}|,
\end{equation}
and 
\begin{equation}\label{NI1d_D}
\hat{\bm{\beta}}^{NI}(\bm{y}, \lambda_{NI}) = \underset{\bm{\beta} \in \mathbb{R}^{n}}{\arg \min} \, \frac{1}{2} ||\bm{y} - \bm{\beta}||_{2}^{2} + \lambda_{NI}\sum_{i=1}^{n-1}|D^{f}\bm{\beta}|_{+},
\end{equation}
respectively, where $D^{f} \in \mathbb{R}^{(n-1)\times n}$ is the following matrix:
\begin{equation}\label{Df1d}
D^{f} = 
\begin{pmatrix}
1 & -1 & 0 & 0 & 0 & \dots & 0 & 0 \\
0 & 1 & -1 & 0 & 0 & \dots & 0 & 0 \\
0 & 0 & 1 & -1 & 0 & \dots & 0 & 0 \\
\vdots & \vdots & \vdots & \vdots  & \vdots  & \vdots & \vdots & \vdots \\
0 & 0 & 0 & 0 & 0 & \dots & 1 & -1
\end{pmatrix}
\end{equation}

Further, in the paper \cite{kim2009ell_1} the authors introduced and studied the following one dimensional approximation
\begin{equation}\label{TF1d}
\hat{\bm{\beta}}^{TF}(\bm{y}, \lambda_{T}) = \underset{\bm{\beta} \in \mathbb{R}^{n}}{\arg \min} \, \frac{1}{2} ||\bm{y} - \bm{\beta}||_{2}^{2} +\lambda_{T} \sum_{i=2}^{n-1}|\beta_{i-1} - 2\beta_{i} + \beta_{i+1}|,
\end{equation}
which is called $\ell_{1}$ trend filtering. The output of $\ell_{1}$ trend filter is a piece-wise linear signal and for a large value of $\lambda_{T}$ it is equivalent to the best $\ell_{2}$ linear approximation of the original signal $\bm{y}$. 

Next, the problem in (\ref{TF1d}) can be re-written as 
\begin{equation}\label{TF1d_D}
\hat{\bm{\beta}}^{TF}(\bm{y}, \lambda_{T}) = \underset{\bm{\beta} \in \mathbb{R}^{n}}{\arg \min} \, \frac{1}{2} ||\bm{y} - \bm{\beta}||_{2}^{2} +\lambda_{T} \sum_{i=2}^{n-1}|D^{t}\bm{\beta}|,
\end{equation}
where $D^{t} \in \mathbb{R}^{(n-2)\times n}$ is the following matrix:
\begin{equation}\label{Mtr1d}
D^{t} = 
\begin{pmatrix}
1 & -2 & 1 & 0 & 0 & \dots & 0 & 0 & 0 \\
0 & 1 & -2 & 1 & 0 & \dots & 0 & 0 & 0 \\
0 & 0 & 1 & -2 & 1 & \dots & 0 & 0 & 0 \\
\vdots & \vdots & \vdots & \vdots  & \vdots & \vdots & \vdots & \vdots & \vdots \\
0 & 0 & 0 & 0 & 0 & \dots & 1  & -2 & 1
\end{pmatrix}
\end{equation}
Note that $D^{f}$ and $D^{t}$ are related in the following way: 
\begin{equation*}
D^{t} =  [D^{f}]_{\{1:(n-2),1:(n-1)\}}D^{f},
\end{equation*}
where the notation $[A]_{\{1:k,1:l\}}$ for a matrix $A \in \mathbb{R}^{n \times m}$ means the the matrix $A$ without last $(n-k)$ rows and $(m-l)$ columns.

In the next subsection we introduce notation and state the problem for the case of estimation on a general grid and graph. 

\subsection{Notation}
Let $G=(V,E)$ be a directed graph, where $V = \{\bm{v}_{1}, \dots, \bm{v}_{n}\}$ is the set of vertices and $E = \{\bm{e}_{1}, \dots, \bm{e}_{m}\}$ is the set of edges. Next, let $D$ denote the oriented incidence matrix for the directed graph $G$. We choose the orientation of $D$ in the following way. Assume that the graph $G$ with $n$ vertices has $m$ edges and we label the vertexes by $\{1, \dots, n\}$ and edges by $\{1, \dots, m\}$. Then $D$ is $m\times n$ matrix with
\begin{equation}\label{grlbls}
D_{i,j} =   \begin{cases}
   1, & \quad \text{if vertex $j$ is the source of edge $i$} , \\
    -1, & \quad \text{if vertex $j$ is the target of edge $i$},\\
    0, & \quad \text{otherwise}.
  \end{cases}
\end{equation}

Now we can introduce the fussed lasso estimator of the signal over the nodes of the graph $G$, which is the following optimization problem:
\begin{equation}\label{FLG}
\hat{\bm{\beta}}^{FL}(\bm{y}, \lambda_{F}) = \underset{\bm{\beta} \in \mathbb{R}^{n}}{\arg \min} \, \frac{1}{2} ||\bm{y} - \bm{\beta}||_{2}^{2} +\lambda_{F}\sum_{(\bm{i},\bm{j})\in E}|\beta_{\bm{i}} - \beta_{\bm{j}}|,
\end{equation}
or, equivalently, using the incidence matrix $D$ of the graph $G$:
\begin{equation}\label{FLD}
\hat{\bm{\beta}}^{FL}(\bm{y}, \lambda_{F}) = {} 
\underset{\bm{\beta} \in \mathbb{R}^{n}}{\arg \min} \, \frac{1}{2} ||\bm{y} - \bm{\beta}||_{2}^{2} + \lambda_{F} ||D\bm{\beta}||_{1}.
\end{equation}

Fused lasso over a grid was first introduced in the signal processing area in \cite{rudin1992nonlinear}. The path solution for the case of a general graph was obtained in \cite{hoefling2010path}. The exact solution for the lattice data in the case of different penalization parameters using taut-string algorithm is given in \cite{barbero2018modular}. Some recent results in the fused lasso estimation of signal over the graphs are \cite{chen2023more, padilla2017dfs, padilla2022variance}.

Further, let $\bm{y} \in \mathbb{R}^{n}$ be a real valued signal indexed by some index set $\mathcal{I} = \{\bm{i}_{1}, \dots, \bm{i}_{n}\}$. Next, we define partial order relation $\preceq$ on $\mathcal{I}$ and isotonic regression with respect to $\preceq$ on $\mathcal{I}$.  

\begin{definition}\label{de1}
A binary relation $\preceq$ on $\mathcal{I}$ is called partial order if 
\begin{itemize}
\item it is reflexive, i.e. $\bm{j}\preceq\bm{j}$ for all $\bm{j} \in \mathcal{I}$;
\item it is transitive, i.e. $\bm{j}_{1}, \bm{j}_{2}, \bm{j}_{3} \in \mathcal{I}$, $\bm{j}_{1} \preceq \bm{j}_{2}$ and $\bm{j}_{2} \preceq \bm{j}_{3}$ imply $\bm{j}_{1} \preceq \bm{j}_{3}$;
\item it is antisymmetric, i.e. $\bm{j}_{1}, \bm{j}_{2} \in \mathcal{I}$, $\bm{j}_{1} \preceq \bm{j}_{2}$ and $\bm{j}_{2} \preceq \bm{j}_{1}$ imply $\bm{j}_{1} = \bm{j}_{2}$.
\end{itemize}
\end{definition}

A vector $\bm{\beta}\in\mathbb{R}^{n}$ indexed by $\mathcal{I}$ is called isotonic with respect to $\preceq$ on $\mathcal{I}$ if $\bm{j}_{1} \preceq \bm{j}_{2}$ implies $\beta_{\bm{j}_{1}} \leq \beta_{\bm{j}_{2}}$. We use the notation $\bm{\mathcal{B}}^{is}$ for the set of all isotonic vectors in $\mathbb{R}^{n}$ with respect to $\preceq$ on $\mathcal{I}$. The set $\bm{\mathcal{B}}^{is}$ is also called isotonic cone in $\mathbb{R}^{n}$. Next, a vector $\bm{\beta}^{I}\in \mathbb{R}^{n}$ is isotonic regression of a vector $\bm{y} \in \mathbb{R}^{n}$ over the set $\mathcal{I}$ with a partial order $\preceq$ on it if 
\begin{eqnarray}\label{Ipo}
\bm{\beta}^{I} = \underset{\bm{\beta} \in \bm{\mathcal{B}}^{is}}{\arg \min} \sum_{\bm{j} \in \mathcal{I}}(\beta_{\bm{j}} - y_{\bm{j}})^{2}.
\end{eqnarray}
In the area of isotonic estimation over partially ordered sets some recent results are \cite{deng2021confidence, han2019isotonic, han2020limit, pananjady2022isotonic}.

Further, for any index set $\mathcal{I}$ with a partial order relation $\preceq$ on $\mathcal{I}$ there exists directed graph $G = (V,E)$, with $V = \mathcal{I}$ (or the set of vertices $V$ is indexed by $\mathcal{I}$) and 
\begin{eqnarray}
E = \{(\bm{j}_{1}, \bm{j}_{2}), \, \text{where} \, (\bm{j_{1}},\bm{j_{2}}) \,\text{ is the ordered pair of vertices from } \, \mathcal{I}\},
\end{eqnarray}
such that a vector $\bm{\beta} \in \mathbb{R}^{n}$ is isotonic with respect to $\preceq$ on $\mathcal{I}$ iff $\beta_{\bm{l_{1}}} \leq \beta_{\bm{l_{2}}}$, given that $E$ contains the chain of edges from $\bm{l}_{1} \in V$ to  $\bm{l}_{2} \in V$. Therefore,  the generalisation of the nearly-isotonic regression defined in (\ref{NI1d}) and (\ref{NI1d_D}) are given by
\begin{equation}\label{NIG}
\hat{\bm{\beta}}^{NI}(\bm{y}, \lambda_{NI}) = \underset{\bm{\beta} \in \mathbb{R}^{n}}{\arg \min} \, \frac{1}{2} ||\bm{y} - \bm{\beta}||_{2}^{2} + \lambda_{NI}\sum_{(\bm{i},\bm{j})\in E}|\beta_{\bm{i}} - \beta_{\bm{j}}|_{+},
\end{equation}
and
\begin{equation}\label{NID}
\hat{\bm{\beta}}^{NI}(\bm{y}, \lambda_{NI}) = {} 
 \underset{\bm{\beta} \in \mathbb{R}^{n}}{\arg \min} \, \frac{1}{2} ||\bm{y} - \bm{\beta}||_{2}^{2} + \lambda_{NI} ||D\bm{\beta}||_{+},
\end{equation}
consequently, where $D$ is the incidence matrix of the directed graph $G$ corresponding to the partial order $\preceq$.

Next, we generalise one dimensional $\ell_{1}$ trend filtering defined in (\ref{TF1d}). Let $G = (V, E)$ be a directed graph with the incidence matrix $D$. We consider two generalisations of one dimensional trend filtering. 

First, the generalisation proposed in \cite{wang2015trend}, which we address simply as general trend filtering, is given by
\begin{equation}\label{GTF}
\begin{aligned}
\hat{\bm{\beta}}^{TF}(\bm{y}, \lambda_{T}) = {} 
 \underset{\bm{\beta} \in \mathbb{R}^{n}}{\arg \min} \, \frac{1}{2} ||\bm{y} - \bm{\beta}||_{2}^{2} + \lambda_{T}\sum_{i=1}^{n}n_{i}|\beta_{\bm{i}} - \frac{1}{n_{i}}\sum_{j:(\bm{i},\bm{j})\in E}\beta_{\bm{j}}|,
\end{aligned}
\end{equation}
with $n_{i}$ the number of neighbours of the vertex $i$,  or, equivalently,
\begin{equation}\label{GTFL}
\hat{\bm{\beta}}^{TF}(\bm{y}, \lambda_{T}) = {} 
 \underset{\bm{\beta} \in \mathbb{R}^{n}}{\arg \min} \, \frac{1}{2} ||\bm{y} - \bm{\beta}||_{2}^{2} +\lambda_{T}||L\bm{\beta}||_{1},
\end{equation}
where $L = D^{T}D$ is the Laplacian of the graph $G$.

Second, if we estimate the signal defined over a multidimensional grid, then we can apply one dimensional $\ell_{1}$ trend filtering defined in (\ref{TF1d}) along each dimension in the grid separately. This approach, which is called called Kronecker trend filter, was first proposed in \cite{kim2009ell_1} and then studied in detail in \cite{sadhanala2021multivariate}. The estimator can also be written in the matrix form as
\begin{equation}\label{KTF}
\hat{\bm{\beta}}^{KTF}(\bm{y}, \lambda_{T}) = {} 
 \underset{\bm{\beta} \in \mathbb{R}^{n}}{\arg \min} \, \frac{1}{2} ||\bm{y} - \bm{\beta}||_{2}^{2} +\lambda_{T}||K||_{1},
\end{equation}
where for $k$-dimensional grid $\{1, \dots, n\}^{k}$ (depending on the numeration of the vertices) the matrix $K$ has the following form
\begin{equation}\label{Kmtrx}
K = \begin{bmatrix}
D^{t} & \otimes & I_{n} & \otimes & \cdots & \otimes & I_{n} \\
I_{n} & \otimes & D^{t} & \otimes & \cdots & \otimes & I_{n} \\
 &  &  &  & \vdots &  &  \\
I_{n} & \otimes & I_{n} & \otimes & \cdots & \otimes & D^{t}
\end{bmatrix}, 
\end{equation}
where $D^{t} \in \mathbb{R}^{(n-2)\times n}$ is the matrix for one dimensional trend filtering defined in (\ref{Mtr1d}), $I_{n}\in \mathbb{R}^{n\times n}$, and $A\otimes B$ denotes the Kronecker product of matrices $A$ and $B$. For more details of Kroneker trend filtering we refer to the original paper \cite{sadhanala2021multivariate}. Note, that in one dimensional case Kronecker trend filtering and general trend filtering are the same.

In Appendix \ref{appA} we provide two examples which clarify general notation given in the subsection above.  




\subsection{General statement of the problem}
Now we are ready to state the general problem studied in this paper. Let $\bm{y} \in \mathbb{R}^{n}$ be a signal indexed by the index set $\mathcal{I}$ with the partial order relation $\preceq$ defined on $\mathcal{I}$. Next, let $G=(V,E)$ be the directed graph corresponding to $\preceq$ on $\mathcal{I}$. 

The nearly-isotonic $\ell_{1}$ trend filter of the signal $\bm{y}$ is:
\begin{equation}\label{NITFG}
\hat{\bm{\beta}}^{NITF}(\bm{y}, \lambda_{NI}, \lambda_{T}) = {} 
 \underset{\bm{\beta} \in \mathbb{R}^{n}}{\arg \min} \, \frac{1}{2} ||\bm{y} - \bm{\beta}||_{2}^{2} + \lambda_{NI} ||D\bm{\beta}||_{+} +\lambda_{T}||\Delta \bm{\beta}||_{1},
\end{equation}
while fused lasso trend filtering on graph $G$ is, consequently, given by
\begin{equation}\label{FLTFL}
\hat{\bm{\beta}}^{FLTF}(\bm{y}, \lambda_{F}, \lambda_{T}) = {} 
 \underset{\bm{\beta} \in \mathbb{R}^{n}}{\arg \min} \, \frac{1}{2} ||\bm{y} - \bm{\beta}||_{2}^{2} + \lambda_{F} ||D\bm{\beta}|| +\lambda_{T}||\Delta \bm{\beta}||_{1},
\end{equation}
where in both cases the matrix $\Delta$ is either
\begin{equation*}\label{}
\Delta = \begin{cases}
    L, &\text{for general trend filtering over the graph}, \\
    K, &\text{for Kronecker trend filtering over the lattice}.
  \end{cases}
\end{equation*}
The combination of different trend filters and combination of trend filters with shape constraints in one dimensional case were discussed in the conclusion of \cite{ramdas2016fast}, where the authors proposed a modification of ADMM algorithm for the numerical solution to different trend filters in one dimensional case.

\subsection{Organisation of the paper}
The paper is organised as follows. In Section \ref{solest} we provide a general solution to the fused $\ell_{1}$ trend filtering over a general graph, show how the fussed lasso trend filter and nearly-isotonic trend filter are related to each other and discuss some properties of the solution.  Next, in Section \ref{compest} we discuss computational aspects of the estimators, do the simulation study and apply the estimator to the real data set. Lastly, in Section \ref{disc} we provide a conclusion and discuss possible further generalisations for the trend filtering problems. The Python implementation of the estimators is available upon request.

\section{Solution to the fused $\ell_{1}$ trend filtering}\label{solest}
First, we consider fussed lasso $\ell_{1}$ trend filter. As mentioned in the Introduction, the combination of different trend filters in one dimensional case (i.e. mixed trend filtering) was proposed in \cite{ramdas2016fast}. Therefore, first, following \cite{ramdas2016fast} let us solve the problem using a regular ADMM algorithm, which has the following update at each iteration:
\begin{eqnarray}\label{ADMM}
 \bm{\beta} & \leftarrow & (I + \rho_{1} D^{T}D + \rho_{2} \Delta^{T}\Delta)^{-1}(\bm{y} + \rho_{1} D^{T}(\bm{\alpha}_{1} + \bm{u}_{1}) + \rho_{2} \Delta^{T}(\bm{\alpha}_{2} + \bm{u}_{2})    )\label{ADMMinv}\\
\bm{\alpha}_{1} & \leftarrow &  S_{\lambda_{F}/\rho_{1}}(D \bm{\beta}  - \bm{u}_{1}),\\
\bm{\alpha}_{2} & \leftarrow &  S_{\lambda_{T}/\rho_{2}}(\Delta \bm{\beta}  - \bm{u}_{2}),\\  
\bm{u}_{1} & \leftarrow & \bm{u}_{1} + \bm{\alpha}_{1} - D\bm{\beta},\\
\bm{u}_{2} & \leftarrow & \bm{u}_{2} + \bm{\alpha}_{2} - \Delta\bm{\beta}\label{ADMM1},
\end{eqnarray}
where $S_{\lambda/\rho}$ is soft-thresholding at the level $\lambda/\rho$. Further, in this paper we consider two ways to implement the first step in (\ref{ADMMinv}): conjugent gradient iteration and sparse Cholesky decomposition. 


Another approach is to solve the problems in (\ref{NITFG}) and (\ref{FLTFL}) via their dual problems. The solutions are given in the next theorem.
\begin{theorem}\label{sol_ftf}
For a fixed data vector $\bm{y} \in \mathbb{R}^{n}$ indexed by the index set $\mathcal{I}$ with the partial order relation $\preceq$ defined on $\mathcal{I}$ and the penalisation parameters $\lambda_{NI}$, $\lambda_{F}$ and $\lambda_{T}$ the solution to the nearly-isotonic trend filtering in (\ref{NITFG}) is given by
\begin{equation}\label{NITF_sol}
\hat{\bm{\beta}}^{NITF}(\bm{y}, \lambda_{NI}, \lambda_{T}) = \bm{y} - D^{T} \hat{\bm{\mu}} - \Delta^{T} \hat{\bm{\nu}}
\end{equation} 
with
\begin{equation}\label{u_solni}
\begin{bmatrix}
\hat{\bm{\mu}} \\
\hat{\bm{\nu}} 
\end{bmatrix}=
\underset{\bm{\mu}, \bm{\nu} }{\arg \min} \, \frac{1}{2}||\bm{y} - D^{T}\bm{\mu} - \Delta^{T} \bm{\nu}||_{2}^{2},
\end{equation}
subject to
\begin{eqnarray*}
\bm{0} \leq& \bm{\mu} &\leq  \lambda_{NI}\bm{1},\\
-\lambda_{T} \bm{1} \leq& \bm{\nu} &\leq  \lambda_{T}\bm{1},\\
\end{eqnarray*}
and solution to the fused lasso trend filtering is
\begin{equation}\label{FLTF_sol}
\hat{\bm{\beta}}^{FLTF}(\bm{y}, \lambda_{F}, \lambda_{T}) = \bm{y} - D^{T} \hat{\bm{\mu}} - \Delta^{T} \hat{\bm{\nu}}
\end{equation} 
with
\begin{equation}\label{u_solfl}
\begin{bmatrix}
\hat{\bm{\mu}} \\
\hat{\bm{\nu}} 
\end{bmatrix}=
\underset{\bm{\mu}, \bm{\nu} }{\arg \min} \, \frac{1}{2}||\bm{y} - D^{T}\bm{\mu} - \Delta^{T} \bm{\nu}||_{2}^{2},
\end{equation}
subject to
\begin{eqnarray*}
-\lambda_{F} \bm{1} \leq \bm{\mu} \leq  \lambda_{F}\bm{1},\\
-\lambda_{T} \bm{1} \leq \bm{\nu} \leq  \lambda_{T}\bm{1},\\
\end{eqnarray*}
where 
\begin{equation*}\label{}
\Delta = \begin{cases}
    L, &\text{for general trend filtering over the graph}, \\
    K, &\text{for Kronecker trend filtering over the lattice},
  \end{cases}
\end{equation*}
$D$ is the incidence matrix of the directed graph $G = (V, E)$ corresponding to $\preceq$ on  $\mathcal{I}$, $\bm{0}$ is zero vector,  $\bm{1}$ is the vector whose all elements are equal to $1$ and the notation $\bm{a} \leq \bm{b}$ for vectors means $a_{i} \leq b_{i}$ for all $i$.
\end{theorem}
\begin{proof}
First, following, for example, the derivations of generalised lasso  in \cite{tibshirani2011solution} and fused nearly-isotonic regression in \cite{pastukhov2022fused} we can write the optimization problem in (\ref{NITFG}) in the following way
\begin{equation*}\label{}
\underset{\bm{\beta}, \bm{z}}{\text{minimize}} \, \frac{1}{2} ||\bm{y} - \bm{\beta}||_{2}^{2} + \lambda_{NI}||\bm{s}||_{+} +  \lambda_{F} ||\bm{t}||_{1}  \quad \text{s. t.} \quad D \bm{\beta} = \bm{z} \quad \text{and } \quad  \Delta \bm{\beta} = \bm{t}.
\end{equation*}

Next, the Lagrangian is given by
\begin{eqnarray}\label{Larg }
L(\bm{\beta}, \bm{s}, \bm{t}, \bm{\mu}, \bm{\nu}) = \frac{1}{2} ||\bm{y} - \bm{\beta}||_{2}^{2} +  \lambda_{NI}||\bm{s}||_{+} + \lambda_{F} ||\bm{t}||_{1} + \\ \nonumber
\bm{\mu}^{T}(D \bm{\beta} - \bm{s}) + \bm{\nu}^{T}(\Delta \bm{\beta} - \bm{t}),
\end{eqnarray}
where $\bm{\mu}$ and $\bm{\nu}$ are dual variables.

Further, note that we have 
\begin{eqnarray}\label{sminni}
\underset{\bm{s}}{\min} \big(\lambda_{NI} ||\bm{s}||_{+}  - \bm{\mu}^{T}\bm{s}\big) &=&
  \begin{cases}
    0, & \quad \text{if} \quad  \bm{0} \leq \bm{\mu} \leq  \lambda_{NI}\bm{1}, \\
    -\infty, & \quad \text{otherwise},
  \end{cases}\\
\underset{\bm{t}}{\min} \big(\lambda_{T} ||\bm{z}||_{1} - \bm{\nu}^{T}\bm{t}\big) &=&
  \begin{cases}
    0, & \quad \text{if} \quad -\lambda_{T} \bm{1} \leq \bm{\nu} \leq \lambda_{T}\bm{1}, \\
    -\infty, & \quad \text{otherwise},\nonumber
  \end{cases}
\end{eqnarray}
and
\begin{eqnarray*}\label{}
\underset{\bm{\beta}}{\min} \big( \frac{1}{2} ||\bm{y} - \bm{\beta}||_{2}^{2} + \bm{\mu}^{T} D \bm{\beta} + \bm{\nu}^{T} \Delta \bm{\beta}\big) &=&\\ -\frac{1}{2}\bm{\mu}^{T}DD^{T}\bm{\mu} &-&\frac{1}{2}\bm{\nu}^{T}\Delta\Delta^{T}\bm{\nu} + \bm{y}^{T}D^{T}\bm{\mu} + \bm{y}^{T}\Delta^{T}\bm{\nu}  =\\
 &-& \frac{1}{2} ||\bm{y} - D^{T}\bm{\mu} - \Delta^{T}\bm{\nu}||_{2}^{2} + \frac{1}{2}\bm{y}^{T}\bm{y}.
\end{eqnarray*}
Therefore, the dual function is
\begin{equation*}\label{}
\begin{aligned}
&g(\nu) ={} \underset{\bm{\beta}, \bm{s}, \bm{t}}{\min} \, L(\bm{\beta}, \bm{s}, \bm{t}, \bm{\mu}, \bm{\nu}) =\\
  &\begin{cases}
    -\frac{1}{2} ||\bm{y} - D^{T}\bm{\mu} - \Delta^{T}\bm{\nu}||_{2}^{2}  + \frac{1}{2}\bm{y}^{T}\bm{y}, & \quad \text{if} \quad  \bm{0} \leq \bm{\mu} \leq \lambda_{NI}\bm{1}, \quad   -\lambda_{T} \bm{1} \leq \bm{\nu} \leq \lambda_{T} \bm{1}, \\
    -\infty, & \quad \text{otherwise},
  \end{cases}
\end{aligned}
\end{equation*}
and, therefore, we have the following dual problem:
\begin{equation*}
\begin{bmatrix}
\hat{\bm{\mu}} \\
\hat{\bm{\nu}} 
\end{bmatrix}=
\underset{\bm{\mu}, \bm{\nu} }{\arg \min} \, \frac{1}{2}||\bm{y} - D^{T}\bm{\mu} - \Delta^{T} \bm{\nu}||_{2}^{2},
\end{equation*}
subject to
\begin{eqnarray*}
\bm{0} \leq& \bm{\mu} &\leq  \lambda_{NI}\bm{1},\\
-\lambda_{T} \bm{1} \leq& \bm{\nu} &\leq  \lambda_{T}\bm{1}.\\
\end{eqnarray*}

Next, the proof for the case of fused lasso $\ell_{1}$ trend filtering is virtually the same with the change of (\ref{sminni}) by
\begin{equation*}\label{}
\underset{\bm{s}}{\min} \big(\lambda_{F} ||\bm{s}|| - \bm{\mu}^{T}\bm{s}\big) =
  \begin{cases}
    0, & \quad \text{if} \quad  -\lambda_{F}\bm{1} \leq \bm{\mu} \leq  \lambda_{F}\bm{1}, \\
    -\infty, & \quad \text{otherwise}.
  \end{cases}
\end{equation*}

Finally, from subgradient equation we obtain 
\begin{equation*}\label{}
\hat{\bm{\beta}} - \bm{y} + D^{T} \hat{\bm{\mu}} + \Delta^{T} \hat{\bm{\nu}} = 0,
\end{equation*} 
which finished the proof.
\end{proof}

Next, we show how nearly-isotonic $\ell_{1}$ trend filter and fused lasso $\ell_{1}$ trend filter are related to each other.
\begin{theorem}\label{rel_ftf}
Assume we have  a fixed data vector $\bm{y} \in \mathbb{R}^{n}$ indexed by the index set $\mathcal{I}$ with the partial order relation $\preceq$ defined on $\mathcal{I}$ (or, equivalently, for a fixed directed graph $G=(V, E)$) and the penalisation parameters $\lambda_{NI}$, $\lambda_{F}$ and $\lambda_{T}$. Then, the estimators $\hat{\bm{\beta}}^{NITF}(\bm{y}, \lambda_{NI}, \lambda_{T})$ and $\hat{\bm{\beta}}^{FLTF}(\bm{y}, \lambda_{F}, \lambda_{T})$ are related in the following way:
\begin{equation*}\label{}
\hat{\bm{\beta}}^{NITF}(\bm{y}, \lambda_{NI}, \lambda_{T}) = \hat{\bm{\beta}}^{FLTF}(\bm{y} - \frac{\lambda_{NI}}{2}D^{T}\bm{1}, \frac{\lambda_{NI}}{2}, \lambda_{T}).
\end{equation*} 

\end{theorem}
\begin{proof}
First, recall that 
\begin{equation*}\label{}
\hat{\bm{\beta}}^{NITF}(\bm{y}, \lambda_{NI}, \lambda_{T}) = \bm{y} - D^{T} \hat{\bm{\mu}} - \Delta^{T} \hat{\bm{\nu}},
\end{equation*} 
with
\begin{equation*}\label{}
\begin{bmatrix}
\hat{\bm{\mu}} \\
\hat{\bm{\nu}} 
\end{bmatrix}=
\underset{\bm{\mu}, \bm{\nu} }{\arg \min} \, \frac{1}{2}||\bm{y} - D^{T}\bm{\mu} - \Delta^{T} \bm{\nu}||_{2}^{2},
\end{equation*}
subject to
\begin{eqnarray*}
\bm{0} \leq& \bm{\mu} &\leq  \lambda_{NI}\bm{1},\\
-\lambda_{T} \bm{1} \leq& \bm{\nu} &\leq  \lambda_{T}\bm{1},\\
\end{eqnarray*}
and 
\begin{equation*}\label{}
\hat{\bm{\beta}}^{FLTF}(\bm{y}, \lambda_{F}, \lambda_{NI}) = \bm{y} - D^{T} \hat{\bm{\xi}} - \Delta^{T} \hat{\bm{\phi}},
\end{equation*} 
with
\begin{equation*}\label{}
\begin{bmatrix}
\hat{\bm{\xi}} \\
\hat{\bm{\phi}} 
\end{bmatrix}=
\underset{\bm{\xi}, \bm{\phi} }{\arg \min} \, \frac{1}{2}||\bm{y} - D^{T}\bm{\xi} - \Delta^{T} \bm{\phi}||_{2}^{2},
\end{equation*}
subject to
\begin{eqnarray*}
-\lambda_{F} \bm{1} \leq \bm{\xi} \leq  \lambda_{F}\bm{1},\\
-\lambda_{T} \bm{1} \leq \bm{\phi} \leq  \lambda_{T}\bm{1}.\\
\end{eqnarray*}

Next, we introduce the variable $\bm{\mu}^{*}  = \bm{\mu} - \frac{\lambda_{NI}}{2}\bm{1}$. Then, we can write for $\hat{\bm{\beta}}^{NITF}$ the following:
\begin{equation*}\label{}
\hat{\bm{\beta}}^{NITF}(\bm{y}, \lambda_{NI}, \lambda_{T}) = \bm{y} - \frac{\lambda_{NI}}{2}D^{T}\bm{1} - D^{T} \hat{\bm{\mu}}^{*} - \Delta^{T} \hat{\bm{\nu}},
\end{equation*} 
where 
\begin{equation*}\label{}
\begin{bmatrix}
\hat{\bm{\mu}}^{*} \\
\hat{\bm{\nu}} 
\end{bmatrix}=
\underset{\bm{\mu}, \bm{\nu} }{\arg \min} \, \frac{1}{2}||\bm{y}  - \frac{\lambda_{NI}}{2} - D^{T} \bm{\mu}^{*} - \Delta^{T} \bm{\nu}||_{2}^{2},
\end{equation*}
subject to
\begin{eqnarray*}
-\frac{\lambda_{NI}}{2}\bm{1} \leq& \bm{\mu}^{*} &\leq  \frac{\lambda_{NI}}{2}\bm{1},\\
-\lambda_{T} \bm{1} \leq& \bm{\nu} &\leq  \lambda_{T}\bm{1},\\
\end{eqnarray*}
which proves the statement of the theorem.
\end{proof}

In the previous theorem we have shown that the solution to the nearly-isotonic $\ell_{1}$ trend filter can be obtained via the solution to the fused lasso $\ell_{1}$ trend filter. Therefore, we can use the original ADMM algorithm in (\ref{ADMM})--(\ref{ADMM1}) to get the solution for the nearly-isotonic case by changing $\bm{y}$ to $\bm{y} - \frac{\lambda_{NI}}{2}$ and $\lambda_{F}$ to $\frac{\lambda_{NI}}{2}$.

Next, even though it is out of the scope of this paper, as a corollary we write the solution to the multi-penalty mixed trend filtering via the solution to the dual problem. Originally the solution to mixed trend filtering was obtained in Section 5.2 of \cite{ramdas2016fast} using ADMM algorithm.
\begin{lemma}\label{m_tf}
Assume we observe  a signal  $\bm{y} \in \mathbb{R}^{n}$ over the nodes of a given directed graph $G = (V, E)$. Next, let us consider the following optimisation problem:
\begin{equation*}\label{}
\hat{\bm{\beta}}^{MTF}(\bm{y}, \lambda_{F}) = {} 
\underset{\bm{\beta} \in \mathbb{R}^{n}}{\arg \min} \, \frac{1}{2} ||\bm{y} - \bm{\beta}||_{2}^{2} +  \sum_{i=1}^{k}\lambda_{i} ||\Delta_{i}\bm{\beta}||_{1} +  \sum_{j=1}^{m}\lambda^{\prime}_{j} ||\Delta^{\prime}_{j}\bm{\beta}||_{+},
\end{equation*}
where $\Delta_{i}$, for $i = 1, \dots, k$ and $\Delta^{\prime}_{j}$, for $j = 1, \dots, m$, are the matrices for the discrete operators of the trend filters of different orders on the graph $G$. Then the solution is given by
\begin{equation*}\label{}
\hat{\bm{\beta}}^{MTF} = \bm{y} - \sum_{i=1}^{k}\Delta_{i}^{T} \hat{\bm{\mu}}_{i}  - \sum_{j=1}^{m}\Delta_{j}^{\prime T} \hat{\bm{\nu}}_{j} 
\end{equation*} 
with
\begin{equation*}\label{}
\begin{bmatrix}
\hat{\bm{\mu}}_{1} \\
\vdots\\
\hat{\bm{\mu}}_{k} \\
\hat{\bm{\nu}}_{1} \\
\vdots\\
\hat{\bm{\nu}}_{m} 
\end{bmatrix}=
\underset{\bm{\mu}, \bm{\nu} }{\arg \min} \, \frac{1}{2}||\bm{y} - \sum_{i=1}^{k}\Delta_{i}^{T} \bm{\mu}_{i}  - \sum_{j=1}^{m}\Delta_{j}^{\prime T} \bm{\nu}_{j} ||_{2}^{2},
\end{equation*}
subject to
\begin{eqnarray*}
-\lambda_{1} \bm{1} \leq &\bm{\mu}_{1}& \leq  \lambda_{1}\bm{1},\\
&\vdots&\\
-\lambda_{k} \bm{1} \leq &\bm{\mu}_{k}& \leq  \lambda_{k}\bm{1},\\
\bm{0} \leq &\bm{\nu}_{1}& \leq  \lambda^{\prime}_{1}\bm{1},\\
&\vdots&\\
\bm{0} \leq &\bm{\nu}_{m}& \leq  \lambda^{\prime}_{m}\bm{1}.
\end{eqnarray*}
\end{lemma}
The proof of the previous lemma is virtually the same as for the Theorem \ref{sol_ftf}.

Next, we prove that both nearly-isotonic and fused lasso $\ell_{1}$ trend filters for both general and Kronecker cases preserve the sum of components of the original signal vector. This important property also holds for isotonic regression over a general partial order set.
\begin{theorem}\label{sum_com}
Assume we observe  a signal  $\bm{y} \in \mathbb{R}^{n}$ over the nodes of a given directed graph $G = (V, E)$. Then for any penalisation parameters $\lambda_{NI}$, $\lambda_{F}$ and $\lambda_{T}$ we have 
\begin{equation*}
\sum_{i=1}^{n} y_{i} = \sum_{i=1}^{n}\hat{\beta}^{NITF}_{i} = \sum_{i=1}^{n}\hat{\beta}^{FLTF}_{i}.
\end{equation*}
\end{theorem}
\begin{proof}
Recall that 
\begin{equation*}\label{}
\hat{\bm{\beta}}^{NITF}(\bm{y}, \lambda_{NI}, \lambda_{T}) = \bm{y} - D^{T} \hat{\bm{\mu}} - \Delta^{T} \hat{\bm{\nu}},
\end{equation*} 
and 
\begin{equation*}\label{}
\hat{\bm{\beta}}^{FLTF}(\bm{y}, \lambda_{F}, \lambda_{NI}) = \bm{y} - D^{T} \hat{\bm{\xi}} - \Delta^{T} \hat{\bm{\phi}}.
\end{equation*} 

Next, note that for any matrix $A\in\mathbb{R}^{k\times m}$ if the sum of the elements of $A$ along every column is zero, then for any $\bm{x} \in \mathbb{R}^{m}$ we have $\sum_{j=1}^{k} [A\bm{x}]_{j} = 0$.

Further, from (\ref{grlbls}) it follows that sum of the elements of $D^{T}$ along every column is zero. Further, in case of genaral filtering, i.e. if $\Delta = L$, the same property holds because $L$ is Laplacian of a directed graph.

Finally, if $\Delta = K$, i.e. Kronecker filtering, this property holds, because $K$ is the matrix with each row having $1$, $-2$, $1$, and the rest values are zeros.

Therefore, we have proved that   
\begin{equation*}
\sum_{i}[D^{T} \hat{\bm{\mu}}]_{i} = \sum_{i}[D^{T} \hat{\bm{\xi}}]_{i} = \sum_{i}[\Delta^{T} \hat{\bm{\nu}}]_{i} = \sum_{i}[\Delta^{T} \hat{\bm{\phi}}]_{i} =0,
\end{equation*}
which finishes the proof of the statement of the theorem.
\end{proof}
This property is important, because it ensures that, for example, these fused trend filters can be used for smoothing of histograms. Also, for example in the applications in physics, if one needs to smooth intensity plot then these filters preserve the total energy.

\section{Computational aspects, simulation study and application to the real data}\label{compest}
First, we address the computational aspects of the estimators. Similarly to \citet{meyer2000degrees, minami2020estimating, pastukhov2022fused},  to test the performance of the estimators we use the following bisigmoid and bicubic functions:
\begin{eqnarray*}
f_{bs}(x^{(1)}, x^{(2)}) &=& \frac{1}{2}\Big(\frac{e^{16x^{(1)} - 8}}{1 + e^{16x^{(1)} - 8}} + \frac{e^{16x^{(2)} - 8}}{1 + e^{16x^{(2)} - 8}}\Big),\\
f_{bc}(x^{(1)}, x^{(2)}) &=& \frac{1}{2}\Big( (2x^{(1)} -1)^{3} + (2x^{(2)} -1)^{3} \Big) + 2,
\end{eqnarray*}
with $x^{(1)} \in [0,1)$ and $x^{(2)} \in [0,1)$, corrupted with i.i.d. noise $\mathcal{N}(0, 0.25)$. Next, we generate homogeneous grid $d\times d$:
\begin{equation*}\label{}
x^{(1)}_{k} = \frac{k-1}{d} \quad \text{and} \quad x^{(2)}_{k} = \frac{k-1}{d},
\end{equation*}
for $k = 1, \dots, d$. The size of the side $d$ varies in $\{ 1\times10^{2}, 2\times 10^{2}, 3\times 10^{2}, 4\times 10^{2},  5\times10^{2} \}$. 

Further, in order to get the numerical solutions to the problems in (\ref{NITFG}) and (\ref{FLTFL}), first, we use two versions of ADMM, one with conjugate gradient solution to the step in (\ref{ADMMinv}) and another with Cholesky decomposition. Also, we obtain the solution by solving the dual problems (\ref{u_solni}) and (\ref{u_solfl}) to  (\ref{NITFG}) and (\ref{FLTFL}), respectively, using recently developed OSQP algorithm, cf. \citet{stellato2020osqp}. We test the performance of general trend filtering and Kronecker filter separately.

Next, we uniformly generate penalisation parameters $\lambda_{F}$ and $\lambda_{T}$ from $U(0,20)$, perform 10 runs and compute average times for each size of the square grid $d$. For ADMM and OSQP algorithms we set $\varepsilon_{abs} = \varepsilon_{rel} = 10^{-3}$ (for the details of the settings in original ADMM we refer to \cite{boyd2011distributed, wahlberg2012admm} and for OSQP we refer to \citet{stellato2020osqp}). 

Figure \ref{compt} below provides these computational times. In the Figure \ref{compt} we use the following abbreviations:
\begin{itemize}
\item FGTF OSQP means fused lasso general trend filtering with the solution obtained by OSQP;
\item FKTF OSQP means, consequently, fused lasso Kronecker trend filtering solved by OSQP;
\item FGTF ADMM CG stands for solution of fused lasso general trend filtering by ADMM with the solution to (\ref{ADMMinv}) by conjugate gradient iteration;
\item FKTF ADMM CG means, consequently, solution of  fused lasso Kronecker trend filtering  by ADMM with the solution to (\ref{ADMMinv}) by conjugate gradient iteration;
\item FGTF ADMM CH refers to solution of fused lasso general trend filtering by ADMM with the solution to (\ref{ADMMinv} with Cholesky decomposition;
\item FKTF ADMM CH stands for fused lasso Kronecker trend filtering solved by ADMM with Cholesky decomposition.
\end{itemize}

All the computations were performed on MacBook Air (Apple M1 chip), 16 GB RAM. From these results one can see that the estimators are computationally feasible for moderate sized data sets. We can also conclude that OSQP is noticeably faster than both versions of ADMM algorithm and the computational performance is almost the same for both versions of trend filters. 

\begin{figure}[!ht] 
  \begin{subfigure}{7cm}
    \centering\includegraphics[scale=0.48]{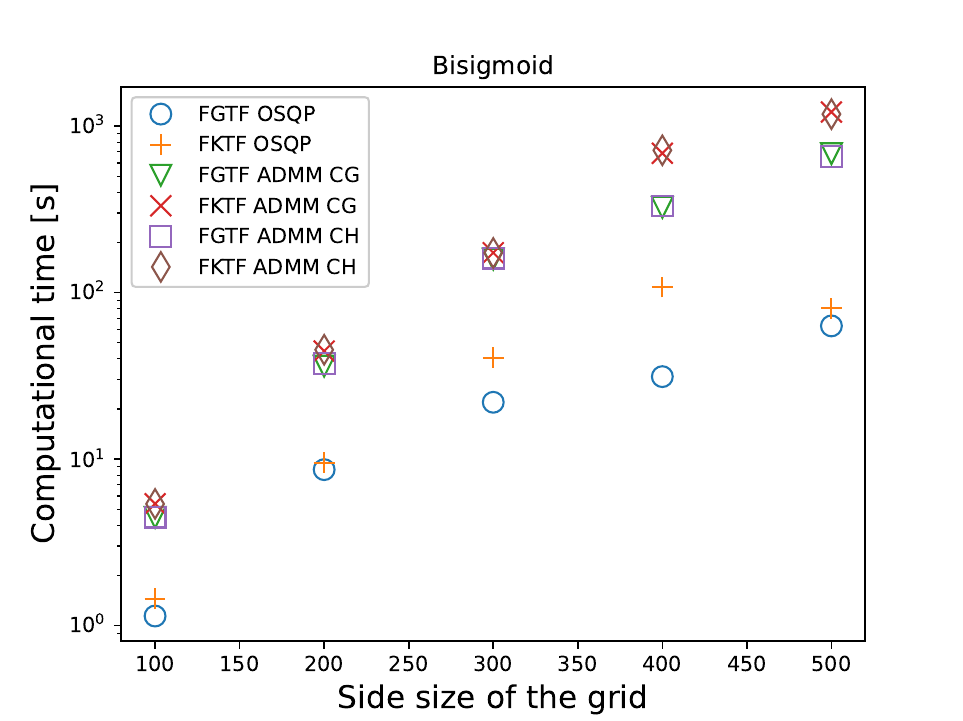}
  \end{subfigure}
  \begin{subfigure}{7cm}
    \centering\includegraphics[scale=0.48]{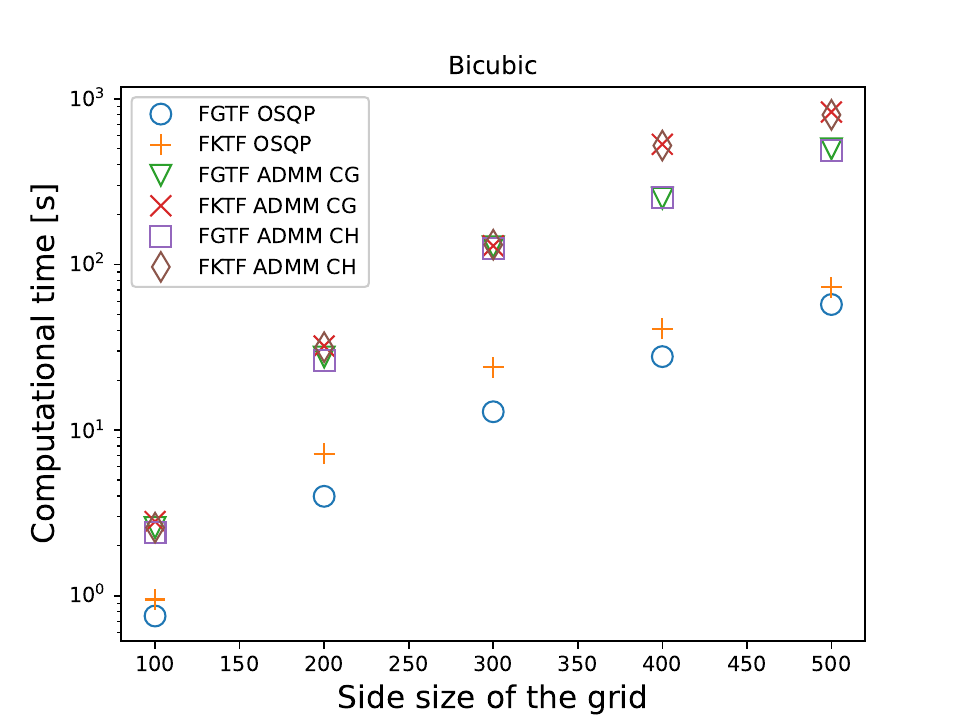}
  \end{subfigure}
   \caption{Computational times vs side size of a square grid for ADMM and OSQP algorithms for fused lasso trend filtering for the cases of general trend filtering and Kronecker trend filtering in two dimensional grid.}\label{compt}
\end{figure}

Next, let us compare general trend filter and Kronecker trend filter without fusion term.  Let us consider the following signal, which is linear in each coordinate and corrupted with some  i.i.d. noise $\mathcal{N}(0, 0.25)$:
\begin{equation*}
y(x^{(1)}, x^{(2)}) = x^{(1)} + x^{(2)} + \varepsilon,
\end{equation*}
measured over homogeneous grid $d\times d$:
\begin{equation*}\label{}
x^{(1)}_{k} = \frac{k-1}{d} \quad \text{and} \quad x^{(2)}_{k} = \frac{k-1}{d}.
\end{equation*}

\begin{figure}[!htb] 
  \begin{subfigure}{6.2cm}
    \centering\includegraphics[scale=0.4]{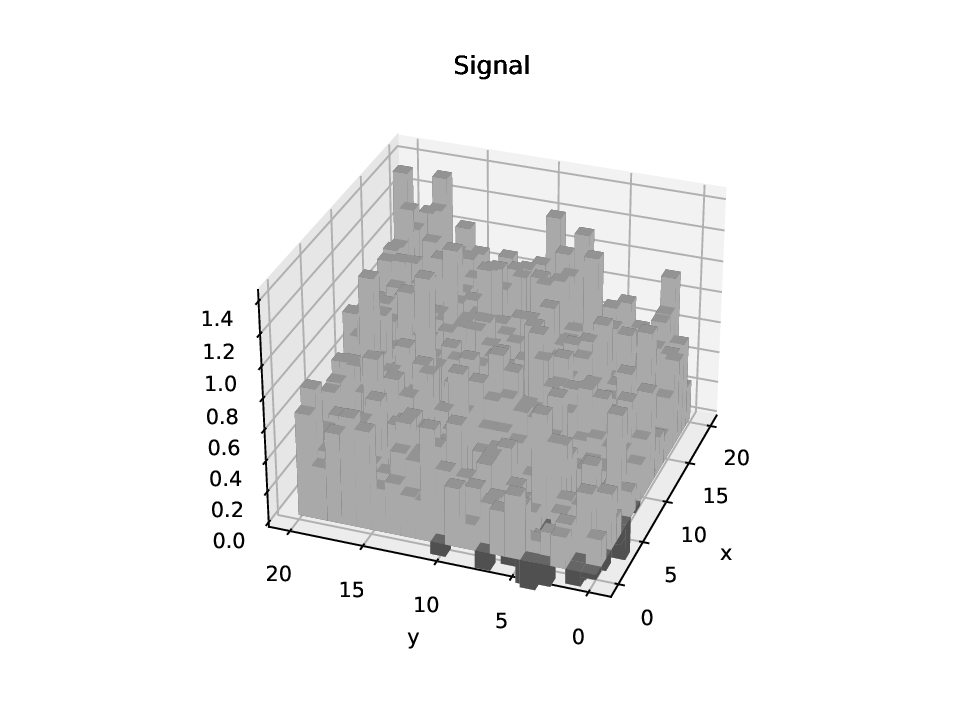}
  \end{subfigure}

  \begin{subfigure}{6.2cm}
    \centering\includegraphics[scale=0.4]{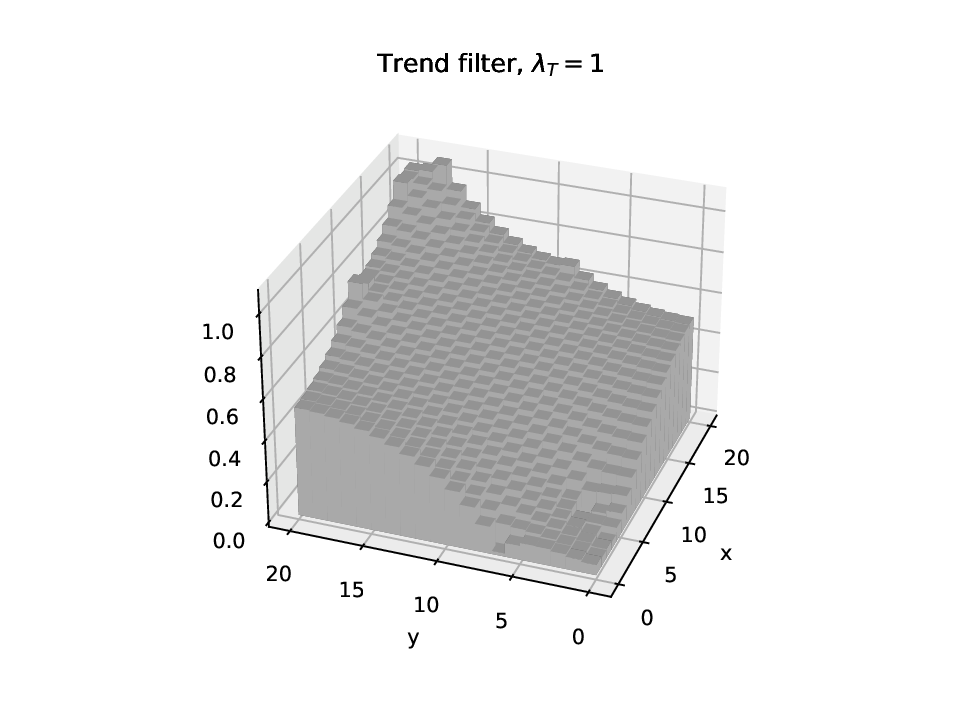}
  \end{subfigure}
  \begin{subfigure}{6.2cm}
    \centering\includegraphics[scale=0.4]{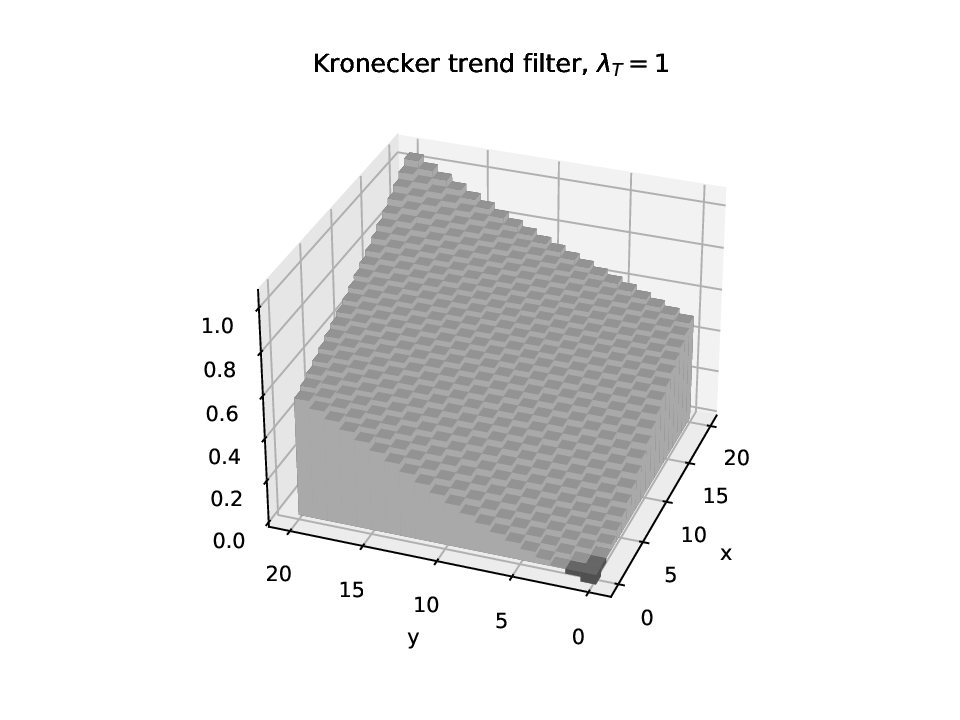}
  \end{subfigure}  
  
   \begin{subfigure}{6.2cm}
    \centering\includegraphics[scale=0.4]{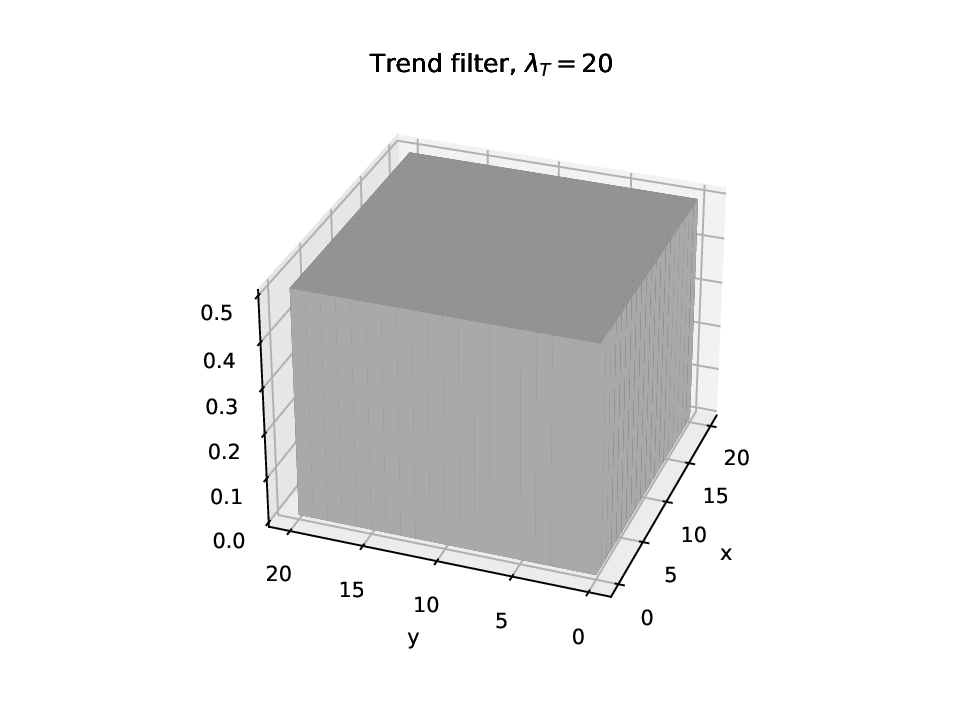}
  \end{subfigure}
  \begin{subfigure}{6.2cm}
    \centering\includegraphics[scale=0.4]{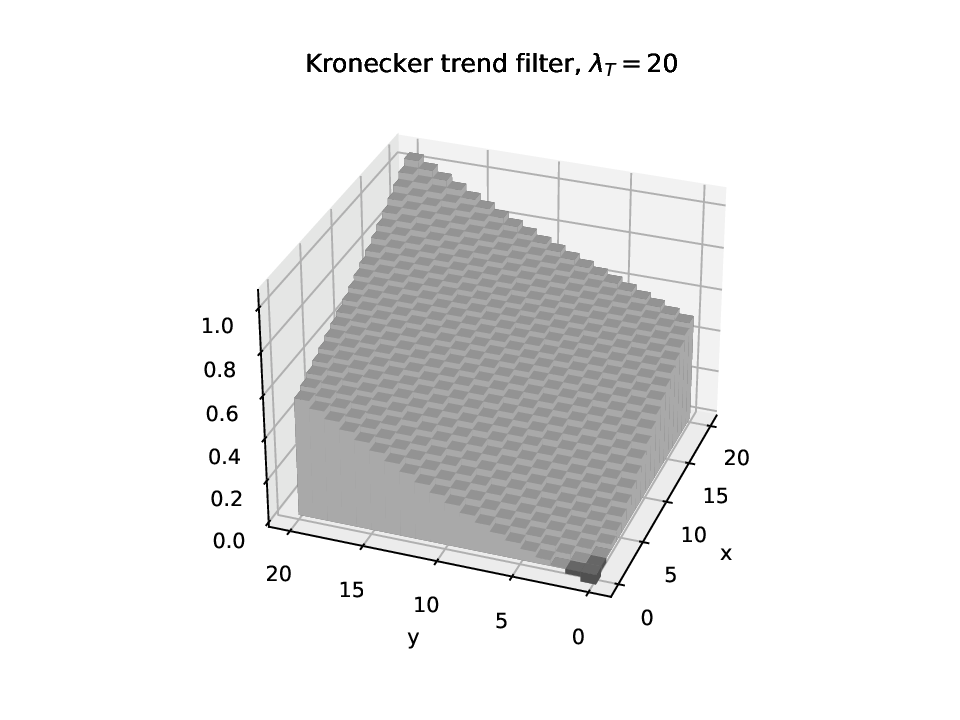}
  \end{subfigure}  
   \caption{Trend filter and Kronecker trend filter for different values of parameter $\lambda_{T}$.}\label{trftrs}
\end{figure} 

From Figure \ref{trftrs} one can see that the general trend filter and Kronecker trend filter behave quite differently. Similarly to $\ell_{1}$ trend filter in one dimension, cf. \cite{kim2009ell_1}, as penalisation paramerter becomes larger, the Kronecker filter tends to the best $\ell_{2}$ linear appromaximation, while the general trend filter tends to the sample average of the whole signal $\bm{y}$ as fused lasso does.

Further, let us consider nearly-isotonic trend filtering. In the paper \cite{pastukhov2022fused} the author studied the combination of nearly-isotonic regression and fused lasso. We consider noisy bisigmoid signal defined above and apply trend filters to the noisy signals.

Figure \ref{nisoktf} displays the signal, isotonic regression and nearly-isotonic trend filtering for different values of penalisation parameters. One can see that as $\lambda_{T}$ becomes larger Kronecker trend filtering makes monotonic approximation more linear. Therefore, nearly-isotononic trend filtering provides a trade-off between monotonic approximation and best linear fit, and it can be used as an alternative approach to the strictly isotonic regression, cf. \cite{wright1978estimating}.

\begin{figure}[!htb] 
  \begin{subfigure}{6.2cm}
    \centering\includegraphics[scale=0.4]{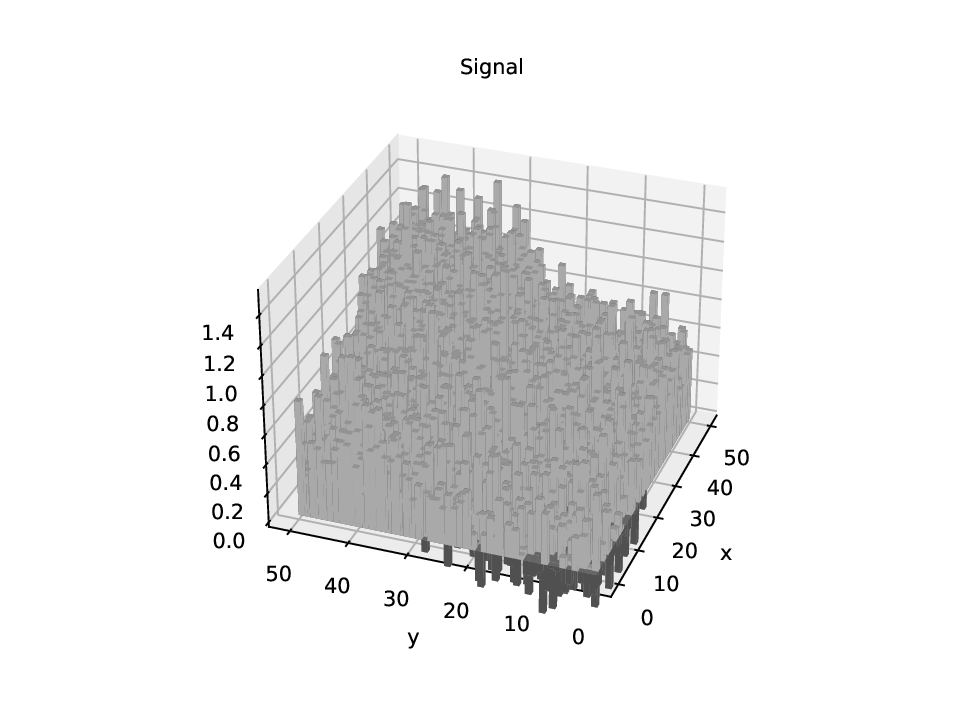}
  \end{subfigure}
  \begin{subfigure}{6.2cm}
    \centering\includegraphics[scale=0.4]{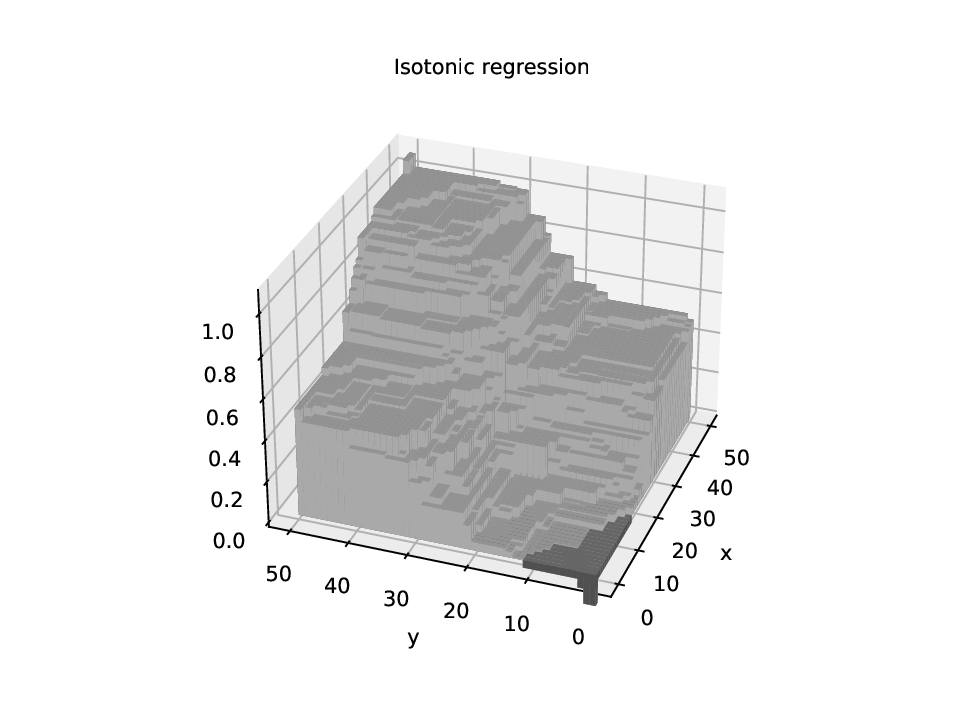}
  \end{subfigure}  
  
   \begin{subfigure}{6.2cm}
    \centering\includegraphics[scale=0.4]{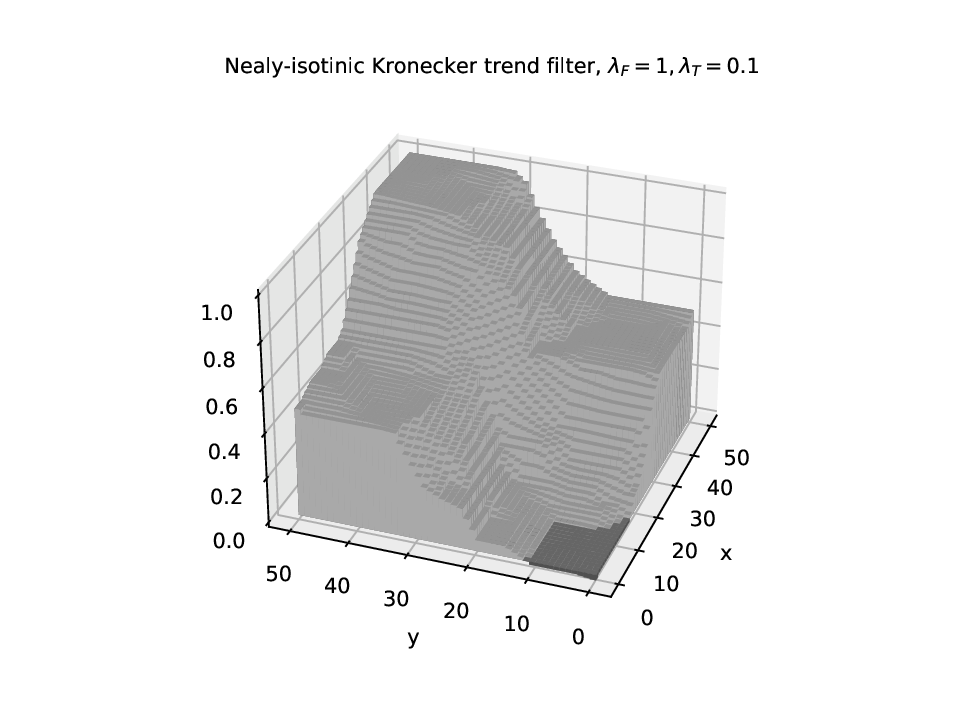}
  \end{subfigure}
  \begin{subfigure}{6.2cm}
    \centering\includegraphics[scale=0.4]{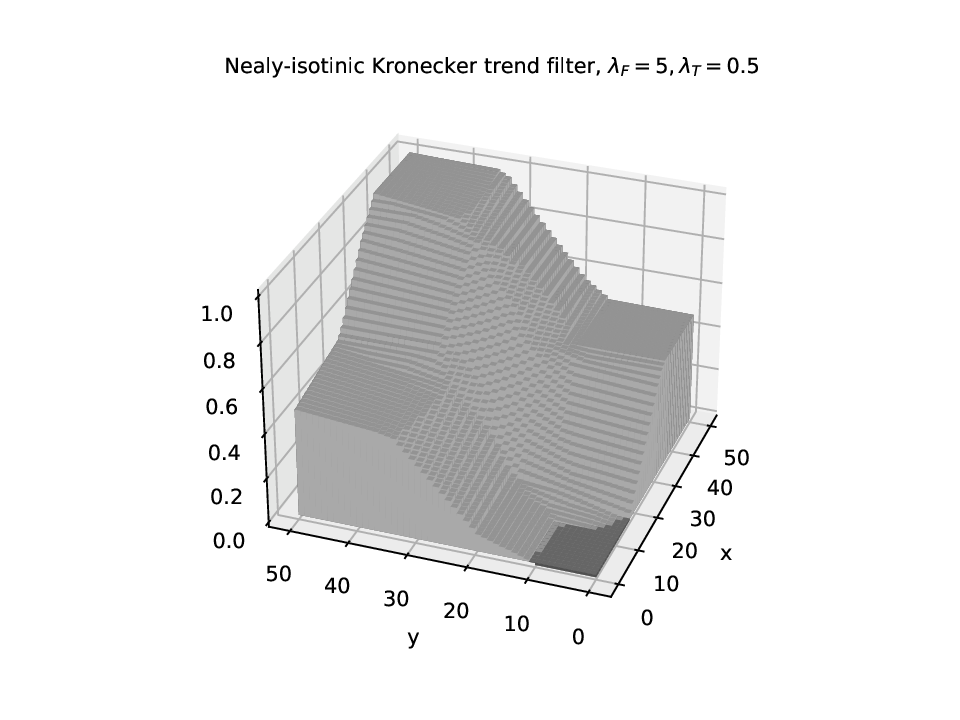}
  \end{subfigure}  
   \caption{Isotonic regression and nearly-isotonic Kronecker trend filter for different values of parameter $\lambda_{NI}$ and $\lambda_{T}$.}\label{nisoktf}
\end{figure} 

Next, we consider the problem of image denoising and inpainting. Figure \ref{chess} presents the original picture of a chessboard, its version without some sets of pixels (white lines) and corrupted with some Normally distributed noise, and the filtered images processed by fused lasso estimator, general trend filter, Kronecker trend filter, fused lasso trend filter and fused lasso Kronecker trend filter.

From Figure \ref{chess} one can see that fussed lasso removes noise quite effectively, though it does not fix missing pixels. On the other hand, both trend filters do impainting quite well, but do not remove noise properly. At the same time, the combination of the fusion and trend filtering can take the best from fusion and filtering, which can be considered as the advantage of the methods studied in this paper.

\begin{figure}[!htb] 
      \begin{subfigure}{6.2cm}
    \centering\includegraphics[scale=0.75]{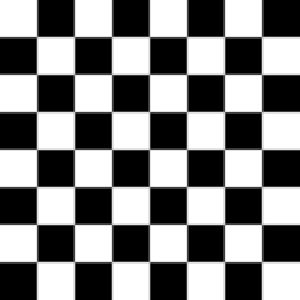}
     \caption{Original image}
  \end{subfigure}
  \begin{subfigure}{6.2cm}
    \centering\includegraphics[scale=0.75]{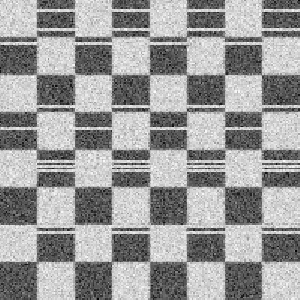}
         \caption{Noizy image}
  \end{subfigure}  
  

   \begin{subfigure}{6.2cm}
    \centering\includegraphics[scale=0.75]{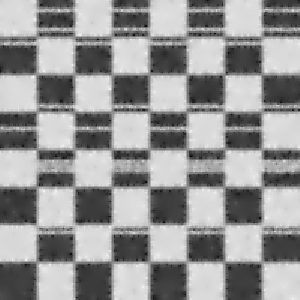}
         \caption{General trend filter for $\lambda_{T} = 50$}
  \end{subfigure}  
     \begin{subfigure}{6.2cm}
    \centering\includegraphics[scale=0.75]{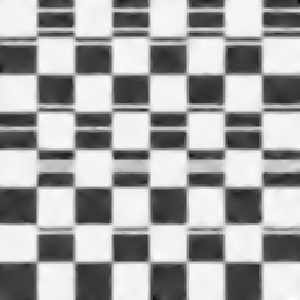}
         \caption{Kronecker trend filter for $\lambda_{T} = 50$}
  \end{subfigure}
  
     \begin{subfigure}{6.2cm}
    \centering\includegraphics[scale=0.75]{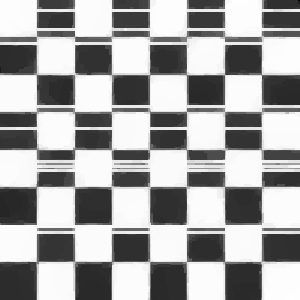}
         \caption{Fused lasso estimator for $\lambda_{F} = 50$}
  \end{subfigure}
  
    \begin{subfigure}{6.2cm}
    \centering\includegraphics[scale=0.75]{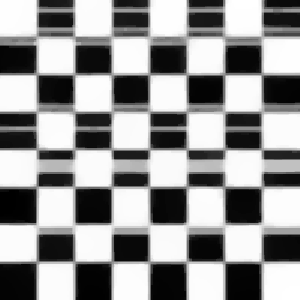}
         \caption{Fused general trend filter for $\lambda_{F} = 50$, $\lambda_{T} = 50$.}
  \end{subfigure}     
     \begin{subfigure}{6.2cm}
    \centering\includegraphics[scale=0.75]{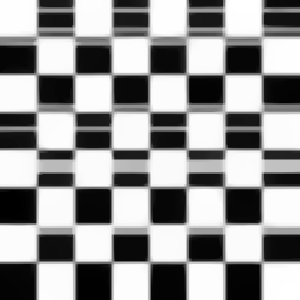}
         \caption{Fused Kronecker trend filter for $\lambda_{F} = 50$, $\lambda_{T} = 50$}
  \end{subfigure}
   \caption{Original image, image with missing pixels and noise, and different image filters.}\label{chess}
\end{figure} 

\newpage

\section{ Conclusion and discussion}\label{disc}
In this paper, we proposed mixing of fused lasso and nearly isotonic regression with general graph filtering and Kronecker filter, which is a useful generalisation of both fusion and trend filtering for estimation of the signal over the nodes of a general directed graph. We provided a computationally feasible approach for big data sets. 

On the simulated data we have shown that nearly-isotonic Kronecker trend filter is potentially interesting and useful approach to make isotonic regression to be strictly monotonic. Also, we gave the example how the combination of fused lasso and trend filtering can improve image denoising and inpainting.

One of the possible interesting direction for the future research is to study fused trend filtering in the case when $\lambda_{F}$ and $\lambda_{T}$ are not the same for every node in the graph. This is important, for example, in the case of sampling over non-homogeneous grid.  Next, analogously to the fused lasso and nearly-isotonic regression, it is important to obtain the estimator of the degrees of freedom for the estimators introduced in this paper.

\newpage

\begin{appendix}
\section{Examples of partial order relation and matrices for the graphs}\label{appA}
In order to clarify notation introduced in the Introduction, let us consider two examples. First assume that $G$ is the following chain graph displayed at Figure~\ref{1dftfgr}.
\begin{figure}[!h]
\begin{center}
\includegraphics[width=3in]{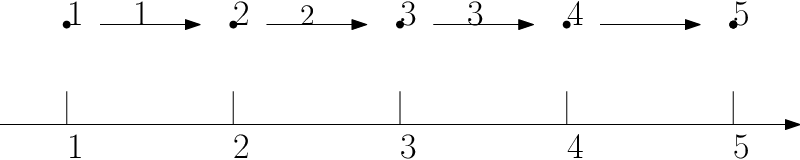}\\
\end{center}
\caption{Graph $G=(V,E)$ which corresponds to the chain graph. \label{1dftfgr}}
\end{figure}
Next, let $\mathcal{I} = \{1, \dots, n\}$, and for $j_{1} \in \mathcal{I}$ and $j_{2} \in \mathcal{I}$  we naturally define $j_{1}\preceq j_{2}$ if $j_{1} \leq j_{2}$. Further, if we let $V = \mathcal{I}$ and $E = \{(i, i+1): i = 1, \dots, n-1 \}$, with $n=5$, then $G = (V, E)$ on Figure~\ref{1dftfgr} is the directed graph which correspond to the one dimensional order relation on $\mathcal{I}$. The incidence matrix $D$ of $G$ is given in (\ref{incm1d}). Also, note that $D = D^{f}$, defined in (\ref{Df1d}).

\begin{equation}\label{incm1d}
D = 
\begin{pmatrix}
1 & -1 & 0 & 0 & 0 \\
0 & 1 & -1 & 0 & 0\\
0 & 0 & 1 & -1 & 0 \\
0 & 0 & 0 &  1 & -1
\end{pmatrix}
\end{equation}

Next, we consider two dimensional case with bimonotonic constraints. The notion of bimonotonicity was first introduced in \citet{beran2010least} and it means the following. Let us consider the index set
\begin{eqnarray*}
\mathcal{I} = \{ \bm{i}= (i^{(1)} ,i^{(2)}): \, i^{(1)}=1,2,\dots, n_{1}, \,  i^{(2)}=1,2,\dots, n_{2}\}
\end{eqnarray*}
with the following order relation $\preceq$ on it: for $\bm{j}_{1}, \bm{j}_{2}\in \mathcal{I}$ we have $\bm{j}_{1} \preceq \bm{j}_{2}$ iff $j^{(1)}_{1} \leq j^{(1)}_{2}$ and $j^{(2)}_{1} \leq j^{(2)}_{2}$. Then, a vector $\bm{\beta}\in\mathbb{R}^{n}$, with $n=n_{1}n_{2}$, indexed by $\mathcal{I}$ is called bimonotone if it is isotonic with respect to bimonotone order $\preceq$ defined on its index $\mathcal{I}$. Further, we define the directed graph $G = (V, E)$ with vertexes $V = \mathcal{I}$, and the edges
\begin{eqnarray*}
\begin{aligned}
E ={}&  \{((l, k),(l, k+1) ): \, 1 \leq l \leq n_{1}, 1 \leq k \leq n_{2} - 1\}\\
\cup \,& \{((l, k),(l+1, k) ): \, 1 \leq l \leq n_{1}-1, 1 \leq k \leq n_{2} \}.
\end{aligned}
\end{eqnarray*}

The oriented graph for $3\times 4$ grid is displayed on Figure \ref{2dftfgr}. The oriented incidence matrix $D \in \mathbb{R}^{17\times 12}$ corresponding to the graph $G=(V,E)$, displayed in Figure (\ref{2dftfgr}), is given in (\ref{2dDmat}), its Laplacian matrix $L = D^{T}D$ is given in (\ref{2dLmat}) and matrix $K$ for Kronecker trend filtering is given in (\ref{2dKmat}).

\begin{figure}[h!]
\begin{center}
\includegraphics[width=3in]{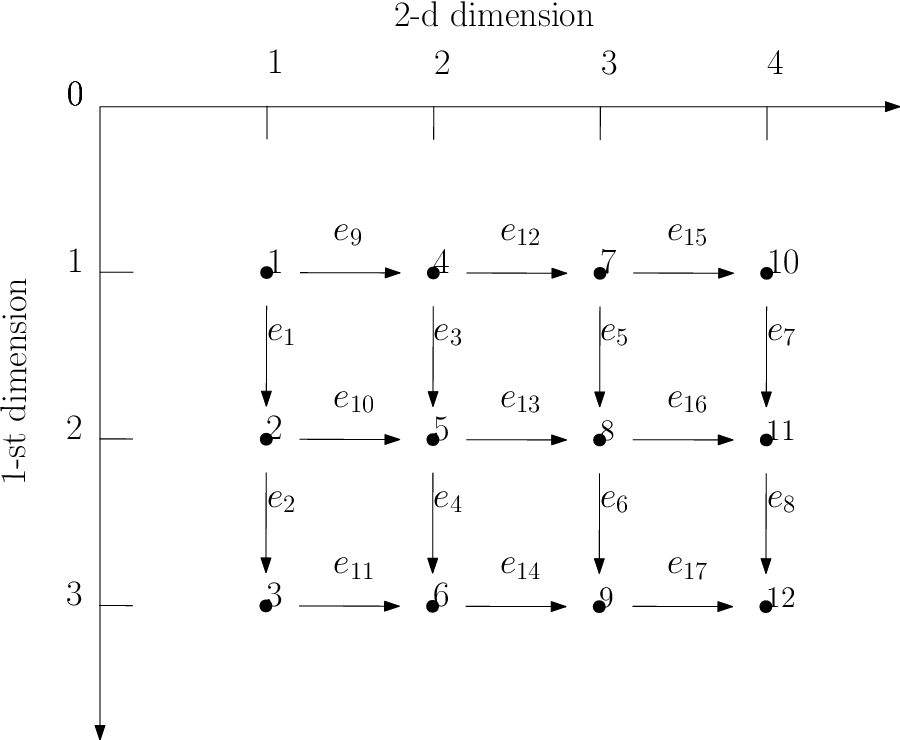}\\
\end{center}
\caption{Graph $G=(V,E)$ which corresponds to the two dimensional equally spaced grid. \label{2dftfgr}}
\end{figure}

\begin{equation}\label{2dDmat}
D = \begin{pmatrix}
1 & -1 & 0 & 0 & 0 &0 & 0 & 0 & 0& 0 & 0& 0\\
0 & 1 & -1 & 0 & 0 &0 & 0 & 0 & 0& 0 & 0& 0\\
0 & 0 & 0 & 1 & -1 & 0 & 0 & 0 & 0& 0 & 0& 0\\
0 & 0 & 0 & 0 & 1 &-1 & 0 & 0 & 0& 0 & 0& 0\\
0 & 0 & 0 & 0 & 0 & 0 & 1 & -1 & 0& 0 & 0& 0\\
0 & 0 & 0 & 0 & 0 & 0 & 0 & 1 & -1& 0& 0& 0\\
0 & 0 & 0 & 0 & 0 & 0 & 0 & 0 & 0 & 1 & -1& 0\\
0 & 0 & 0 &0 & 0 & 0 & 0 & 0 & 0 & 0 & 1& -1\\
1 & 0 & 0 & -1 & 0 & 0 & 0 & 0 & 0 & 0& 0& 0\\
0 & 1 & 0 & 0 & -1 & 0 & 0 & 0 & 0 & 0& 0& 0\\
0 & 0 & 1 & 0 & 0 & -1 & 0 & 0 & 0 & 0& 0& 0\\
0 & 0 & 0 & 1 & 0 & 0 & -1 & 0 & 0 & 0& 0& 0\\
0 & 0 & 0 & 0 & 1 & 0 & 0 & -1 & 0 & 0& 0& 0\\
0 & 0 & 0 & 0 & 0 & 1 & 0 & 0 & -1 & 0& 0& 0\\
0 & 0 & 0 & 0 & 0 & 0 & 1 & 0 & 0 & -1& 0& 0\\
0 & 0 & 0 & 0 & 0 & 0 & 0 & 1 & 0 &0& 0& -1
\end{pmatrix}
\end{equation}

\begin{equation}\label{2dLmat}
L = \begin{pmatrix}
2 & -1 & 0 & -1 & 0 &0 & 0 & 0 &0&0&0&0 \\
-1 & 3 & -1 & 0 & -1 &0 & 0 & 0 &0&0&0&0\\
0 & -1 & 2 & 0 & 0 & -1 & 0 & 0 &0&0&0&0 \\
-1 & 0 & 0 & 3 & -1 &0 & -1 & 0 &0 &0&0&0 \\
0 & -1 & 0 & -1 & 4 & -1 & 0 &-1 & 0 &0&0&0\\
0 & 0 & -1 &0 & -1 & 3 & 0 &0 & -1 &0&0&0 \\
0 & 0 & 0 &-1 & 0 & 0 & 3 &-1 & 0 &-1&0&0 \\
0 & 0 & 0 &0 & -1 & 0 & -1 &4 & -1 &0&-1&0 \\
0 & 0 & 0 &0 & 0 & -1 & 0 &-1 & 3 &0&0&-1 \\
0 & 0 & 0 &0 & 0 & 0 & -1 &0 & 0 &2&-1&0 \\
0 & 0 & 0 &0 & 0 & 0 & 0 &-1 & 0 &-1&3&-1 \\
0 & 0 & 0 &0 & 0 & 0 & 0 &-1 & 0 &0&-1&2 
\end{pmatrix}
\end{equation}

\begin{equation}\label{2dKmat}
K = \begin{pmatrix}
-1 & 2 & -1 & 0 & 0 &0 & 0 & 0 &0&0&0&0 \\
0 & 0 & 0 & -1 & 2 &-1 & 0 & 0 &0&0&0&0\\
0 & 0 & 0 & 0 & 0 & 0 & -1 & 2 &-1&0&0&0 \\
0 & 0 & 0 & 0 & 0 &0 & 0 & 0 &0 &-1&2&-1 \\
-1 & 0 & 0 & 2 & 0 & 0 & -1 &0 & 0 &0&0&0\\
0 & -1 & 0 &0 & 2 & 0 & 0 &-1 & 0 &0&0&0 \\
0 & 0 & -1 &0 & 0 & 2 & 0 &0 & -1 & 0 &0&0 \\
0 & 0 & 0 &-1 & 0 & 0 & 2 &0 & 0 &-1&0&0 \\
0 & 0 & 0 &0 & -1 & 0 & 0 &2 & 0 &0&-1&0 \\
0 & 0 & 0 &0 & 0 & -1 & 0 &0 & 2 &0&0&-1 
\end{pmatrix}
\end{equation}




\end{appendix}

\begin{acks}[Acknowledgments]
This work was partially supported by the Wallenberg AI, Autonomous Systems and Software Program (WASP) funded by the Knut and Alice Wallenberg Foundataion. 
\end{acks}

\bibliographystyle{imsart-nameyear} 
\bibliography{sample}

\begin{thebibliography}{28}

\bibitem[\protect\citeauthoryear{Barbero and Sra}{2018}]{barbero2018modular}
\begin{barticle}[author]
\bauthor{\bsnm{Barbero},~\bfnm{Alvaro}\binits{A.}} \AND
  \bauthor{\bsnm{Sra},~\bfnm{Suvrit}\binits{S.}}
(\byear{2018}).
\btitle{Modular proximal optimization for multidimensional total-variation
  regularization}.
\bjournal{The Journal of Machine Learning Research}
\bvolume{19}
\bpages{2232--2313}.
\end{barticle}
\endbibitem

\bibitem[\protect\citeauthoryear{Beran and D{\"u}mbgen}{2010}]{beran2010least}
\begin{barticle}[author]
\bauthor{\bsnm{Beran},~\bfnm{Rudolf}\binits{R.}} \AND
  \bauthor{\bsnm{D{\"u}mbgen},~\bfnm{Lutz}\binits{L.}}
(\byear{2010}).
\btitle{Least squares and shrinkage estimation under bimonotonicity
  constraints}.
\bjournal{Statistics and computing}
\bvolume{20}
\bpages{177--189}.
\end{barticle}
\endbibitem

\bibitem[\protect\citeauthoryear{Boyd et~al.}{2011}]{boyd2011distributed}
\begin{barticle}[author]
\bauthor{\bsnm{Boyd},~\bfnm{Stephen}\binits{S.}},
  \bauthor{\bsnm{Parikh},~\bfnm{Neal}\binits{N.}},
  \bauthor{\bsnm{Chu},~\bfnm{Eric}\binits{E.}},
  \bauthor{\bsnm{Peleato},~\bfnm{Borja}\binits{B.}},
  \bauthor{\bsnm{Eckstein},~\bfnm{Jonathan}\binits{J.}} \betal{et~al.}
(\byear{2011}).
\btitle{Distributed optimization and statistical learning via the alternating
  direction method of multipliers}.
\bjournal{Foundations and Trends{\textregistered} in Machine learning}
\bvolume{3}
\bpages{1--122}.
\end{barticle}
\endbibitem

\bibitem[\protect\citeauthoryear{Brunk et~al.}{1972}]{brunk1972statistical}
\begin{btechreport}[author]
\bauthor{\bsnm{Brunk},~\bfnm{HD}\binits{H.}},
  \bauthor{\bsnm{Barlow},~\bfnm{Richard~E}\binits{R.~E.}},
  \bauthor{\bsnm{Bartholomew},~\bfnm{Daniel~J}\binits{D.~J.}} \AND
  \bauthor{\bsnm{Bremner},~\bfnm{James~M}\binits{J.~M.}}
(\byear{1972}).
\btitle{Statistical inference under order restrictions.(the theory and
  application of isotonic regression)}
\btype{Technical Report},
\bpublisher{Missouri Univ Columbia Dept of Statistics}.
\end{btechreport}
\endbibitem

\bibitem[\protect\citeauthoryear{Chen, Jewell and Witten}{2023}]{chen2023more}
\begin{barticle}[author]
\bauthor{\bsnm{Chen},~\bfnm{Yiqun}\binits{Y.}},
  \bauthor{\bsnm{Jewell},~\bfnm{Sean}\binits{S.}} \AND
  \bauthor{\bsnm{Witten},~\bfnm{Daniela}\binits{D.}}
(\byear{2023}).
\btitle{More powerful selective inference for the graph fused lasso}.
\bjournal{Journal of Computational and Graphical Statistics}
\bvolume{32}
\bpages{577--587}.
\end{barticle}
\endbibitem

\bibitem[\protect\citeauthoryear{Deng, Han and
  Zhang}{2021}]{deng2021confidence}
\begin{barticle}[author]
\bauthor{\bsnm{Deng},~\bfnm{Hang}\binits{H.}},
  \bauthor{\bsnm{Han},~\bfnm{Qiyang}\binits{Q.}} \AND
  \bauthor{\bsnm{Zhang},~\bfnm{Cun-Hui}\binits{C.-H.}}
(\byear{2021}).
\btitle{Confidence intervals for multiple isotonic regression and other
  monotone models}.
\bjournal{The Annals of Statistics}
\bvolume{49}
\bpages{2021--2052}.
\end{barticle}
\endbibitem

\bibitem[\protect\citeauthoryear{Han and Zhang}{2020}]{han2020limit}
\begin{barticle}[author]
\bauthor{\bsnm{Han},~\bfnm{Qiyang}\binits{Q.}} \AND
  \bauthor{\bsnm{Zhang},~\bfnm{Cun-Hui}\binits{C.-H.}}
(\byear{2020}).
\btitle{Limit distribution theory for block estimators in multiple isotonic
  regression}.
\bjournal{The Annals of Statistics}
\bvolume{48}
\bpages{3251-3282}.
\end{barticle}
\endbibitem

\bibitem[\protect\citeauthoryear{Han et~al.}{2019}]{han2019isotonic}
\begin{barticle}[author]
\bauthor{\bsnm{Han},~\bfnm{Qiyang}\binits{Q.}},
  \bauthor{\bsnm{Wang},~\bfnm{Tengyao}\binits{T.}},
  \bauthor{\bsnm{Chatterjee},~\bfnm{Sabyasachi}\binits{S.}} \AND
  \bauthor{\bsnm{Samworth},~\bfnm{Richard~J}\binits{R.~J.}}
(\byear{2019}).
\btitle{Isotonic regression in general dimensions}.
\bjournal{The Annals of Statistics}
\bvolume{47}
\bpages{2440--2471}.
\end{barticle}
\endbibitem

\bibitem[\protect\citeauthoryear{Hoefling}{2010}]{hoefling2010path}
\begin{barticle}[author]
\bauthor{\bsnm{Hoefling},~\bfnm{Holger}\binits{H.}}
(\byear{2010}).
\btitle{A path algorithm for the fused lasso signal approximator}.
\bjournal{Journal of Computational and Graphical Statistics}
\bvolume{19}
\bpages{984--1006}.
\end{barticle}
\endbibitem

\bibitem[\protect\citeauthoryear{Kim et~al.}{2009}]{kim2009ell_1}
\begin{barticle}[author]
\bauthor{\bsnm{Kim},~\bfnm{Seung-Jean}\binits{S.-J.}},
  \bauthor{\bsnm{Koh},~\bfnm{Kwangmoo}\binits{K.}},
  \bauthor{\bsnm{Boyd},~\bfnm{Stephen}\binits{S.}} \AND
  \bauthor{\bsnm{Gorinevsky},~\bfnm{Dimitry}\binits{D.}}
(\byear{2009}).
\btitle{$\ell_1$ trend filtering}.
\bjournal{SIAM Review}
\bvolume{51}
\bpages{339--360}.
\end{barticle}
\endbibitem

\bibitem[\protect\citeauthoryear{Meyer and Woodroofe}{2000}]{meyer2000degrees}
\begin{barticle}[author]
\bauthor{\bsnm{Meyer},~\bfnm{Mary}\binits{M.}} \AND
  \bauthor{\bsnm{Woodroofe},~\bfnm{Michael}\binits{M.}}
(\byear{2000}).
\btitle{On the degrees of freedom in shape-restricted regression}.
\bjournal{The Annals of Statistics}
\bvolume{28}
\bpages{1083--1104}.
\end{barticle}
\endbibitem

\bibitem[\protect\citeauthoryear{Minami}{2020}]{minami2020estimating}
\begin{barticle}[author]
\bauthor{\bsnm{Minami},~\bfnm{Kentaro}\binits{K.}}
(\byear{2020}).
\btitle{Estimating piecewise monotone signals}.
\bjournal{Electronic Journal of Statistics}
\bvolume{14}
\bpages{1508--1576}.
\end{barticle}
\endbibitem

\bibitem[\protect\citeauthoryear{Padilla}{2022}]{padilla2022variance}
\begin{barticle}[author]
\bauthor{\bsnm{Padilla},~\bfnm{Oscar Hernan~Madrid}\binits{O.~H.~M.}}
(\byear{2022}).
\btitle{Variance estimation in graphs with the fused lasso}.
\bjournal{arXiv preprint arXiv:2207.12638}.
\end{barticle}
\endbibitem

\bibitem[\protect\citeauthoryear{Padilla et~al.}{2017}]{padilla2017dfs}
\begin{barticle}[author]
\bauthor{\bsnm{Padilla},~\bfnm{Oscar Hernan~Madrid}\binits{O.~H.~M.}},
  \bauthor{\bsnm{Sharpnack},~\bfnm{James}\binits{J.}},
  \bauthor{\bsnm{Scott},~\bfnm{James~G}\binits{J.~G.}} \AND
  \bauthor{\bsnm{Tibshirani},~\bfnm{Ryan~J}\binits{R.~J.}}
(\byear{2017}).
\btitle{The DFS Fused Lasso: Linear-Time Denoising over General Graphs.}
\bjournal{J. Mach. Learn. Res.}
\bvolume{18}
\bpages{1--36}.
\end{barticle}
\endbibitem

\bibitem[\protect\citeauthoryear{Pananjady and
  Samworth}{2022}]{pananjady2022isotonic}
\begin{barticle}[author]
\bauthor{\bsnm{Pananjady},~\bfnm{Ashwin}\binits{A.}} \AND
  \bauthor{\bsnm{Samworth},~\bfnm{Richard~J}\binits{R.~J.}}
(\byear{2022}).
\btitle{Isotonic regression with unknown permutations: Statistics, computation
  and adaptation}.
\bjournal{The Annals of Statistics}
\bvolume{50}
\bpages{324--350}.
\end{barticle}
\endbibitem

\bibitem[\protect\citeauthoryear{Pastukhov}{2022}]{pastukhov2022fused}
\begin{bmisc}[author]
\bauthor{\bsnm{Pastukhov},~\bfnm{Vladimir}\binits{V.}}
(\byear{2022}).
\btitle{Fused Lasso Nearly Isotonic Signal Approximation in General
  Dimensions}.
\end{bmisc}
\endbibitem

\bibitem[\protect\citeauthoryear{Ramdas and Tibshirani}{2016}]{ramdas2016fast}
\begin{barticle}[author]
\bauthor{\bsnm{Ramdas},~\bfnm{Aaditya}\binits{A.}} \AND
  \bauthor{\bsnm{Tibshirani},~\bfnm{Ryan~J}\binits{R.~J.}}
(\byear{2016}).
\btitle{Fast and flexible ADMM algorithms for trend filtering}.
\bjournal{Journal of Computational and Graphical Statistics}
\bvolume{25}
\bpages{839--858}.
\end{barticle}
\endbibitem

\bibitem[\protect\citeauthoryear{Rinaldo}{2009}]{rinaldo2009properties}
\begin{barticle}[author]
\bauthor{\bsnm{Rinaldo},~\bfnm{Alessandro}\binits{A.}}
(\byear{2009}).
\btitle{Properties and refinements of the fused lasso}.
\bjournal{The Annals of Statistics}
\bvolume{37}
\bpages{2922--2952}.
\end{barticle}
\endbibitem

\bibitem[\protect\citeauthoryear{Robertson, Wright and
  Dykstra}{1988}]{robertson1988order}
\begin{bbook}[author]
\bauthor{\bsnm{Robertson},~\bfnm{Tim}\binits{T.}},
  \bauthor{\bsnm{Wright},~\bfnm{F.~T.}\binits{F.~T.}} \AND
  \bauthor{\bsnm{Dykstra},~\bfnm{R.~L.}\binits{R.~L.}}
(\byear{1988}).
\btitle{Order restricted statistical inference}.
\bpublisher{Wiley}.
\end{bbook}
\endbibitem

\bibitem[\protect\citeauthoryear{Rudin, Osher and
  Fatemi}{1992}]{rudin1992nonlinear}
\begin{barticle}[author]
\bauthor{\bsnm{Rudin},~\bfnm{Leonid~I}\binits{L.~I.}},
  \bauthor{\bsnm{Osher},~\bfnm{Stanley}\binits{S.}} \AND
  \bauthor{\bsnm{Fatemi},~\bfnm{Emad}\binits{E.}}
(\byear{1992}).
\btitle{Nonlinear total variation based noise removal algorithms}.
\bjournal{Physica D: nonlinear phenomena}
\bvolume{60}
\bpages{259--268}.
\end{barticle}
\endbibitem

\bibitem[\protect\citeauthoryear{Sadhanala
  et~al.}{2021}]{sadhanala2021multivariate}
\begin{barticle}[author]
\bauthor{\bsnm{Sadhanala},~\bfnm{Veeranjaneyulu}\binits{V.}},
  \bauthor{\bsnm{Wang},~\bfnm{Yu-Xiang}\binits{Y.-X.}},
  \bauthor{\bsnm{Hu},~\bfnm{Addison~J}\binits{A.~J.}} \AND
  \bauthor{\bsnm{Tibshirani},~\bfnm{Ryan~J}\binits{R.~J.}}
(\byear{2021}).
\btitle{Multivariate trend filtering for lattice data}.
\bjournal{arXiv preprint arXiv:2112.14758}.
\end{barticle}
\endbibitem

\bibitem[\protect\citeauthoryear{Stellato et~al.}{2020}]{stellato2020osqp}
\begin{barticle}[author]
\bauthor{\bsnm{Stellato},~\bfnm{Bartolomeo}\binits{B.}},
  \bauthor{\bsnm{Banjac},~\bfnm{Goran}\binits{G.}},
  \bauthor{\bsnm{Goulart},~\bfnm{Paul}\binits{P.}},
  \bauthor{\bsnm{Bemporad},~\bfnm{Alberto}\binits{A.}} \AND
  \bauthor{\bsnm{Boyd},~\bfnm{Stephen}\binits{S.}}
(\byear{2020}).
\btitle{OSQP: An operator splitting solver for quadratic programs}.
\bjournal{Mathematical Programming Computation}
\bvolume{12}
\bpages{637--672}.
\end{barticle}
\endbibitem

\bibitem[\protect\citeauthoryear{Tibshirani, Hoefling and
  Tibshirani}{2011}]{tibshirani2011nearly}
\begin{barticle}[author]
\bauthor{\bsnm{Tibshirani},~\bfnm{Ryan~J}\binits{R.~J.}},
  \bauthor{\bsnm{Hoefling},~\bfnm{Holger}\binits{H.}} \AND
  \bauthor{\bsnm{Tibshirani},~\bfnm{Robert}\binits{R.}}
(\byear{2011}).
\btitle{Nearly-isotonic regression}.
\bjournal{Technometrics}
\bvolume{53}
\bpages{54--61}.
\end{barticle}
\endbibitem

\bibitem[\protect\citeauthoryear{Tibshirani and
  Taylor}{2011}]{tibshirani2011solution}
\begin{barticle}[author]
\bauthor{\bsnm{Tibshirani},~\bfnm{Ryan~J}\binits{R.~J.}} \AND
  \bauthor{\bsnm{Taylor},~\bfnm{Jonathan}\binits{J.}}
(\byear{2011}).
\btitle{The solution path of the generalized lasso}.
\bjournal{The Annals of Statistics}
\bvolume{39}
\bpages{1335--1371}.
\end{barticle}
\endbibitem

\bibitem[\protect\citeauthoryear{Tibshirani
  et~al.}{2005}]{tibshirani2005sparsity}
\begin{barticle}[author]
\bauthor{\bsnm{Tibshirani},~\bfnm{Robert}\binits{R.}},
  \bauthor{\bsnm{Saunders},~\bfnm{Michael}\binits{M.}},
  \bauthor{\bsnm{Rosset},~\bfnm{Saharon}\binits{S.}},
  \bauthor{\bsnm{Zhu},~\bfnm{Ji}\binits{J.}} \AND
  \bauthor{\bsnm{Knight},~\bfnm{Keith}\binits{K.}}
(\byear{2005}).
\btitle{Sparsity and smoothness via the fused lasso}.
\bjournal{Journal of the Royal Statistical Society: Series B (Statistical
  Methodology)}
\bvolume{67}
\bpages{91--108}.
\end{barticle}
\endbibitem

\bibitem[\protect\citeauthoryear{Wahlberg et~al.}{2012}]{wahlberg2012admm}
\begin{barticle}[author]
\bauthor{\bsnm{Wahlberg},~\bfnm{Bo}\binits{B.}},
  \bauthor{\bsnm{Boyd},~\bfnm{Stephen}\binits{S.}},
  \bauthor{\bsnm{Annergren},~\bfnm{Mariette}\binits{M.}} \AND
  \bauthor{\bsnm{Wang},~\bfnm{Yang}\binits{Y.}}
(\byear{2012}).
\btitle{An ADMM algorithm for a class of total variation regularized estimation
  problems}.
\bjournal{IFAC Proceedings Volumes}
\bvolume{45}
\bpages{83--88}.
\end{barticle}
\endbibitem

\bibitem[\protect\citeauthoryear{Wang et~al.}{2015}]{wang2015trend}
\begin{binproceedings}[author]
\bauthor{\bsnm{Wang},~\bfnm{Yu-Xiang}\binits{Y.-X.}},
  \bauthor{\bsnm{Sharpnack},~\bfnm{James}\binits{J.}},
  \bauthor{\bsnm{Smola},~\bfnm{Alex}\binits{A.}} \AND
  \bauthor{\bsnm{Tibshirani},~\bfnm{Ryan}\binits{R.}}
(\byear{2015}).
\btitle{Trend filtering on graphs}.
In \bbooktitle{Artificial Intelligence and Statistics}
\bpages{1042--1050}.
\bpublisher{PMLR}.
\end{binproceedings}
\endbibitem

\bibitem[\protect\citeauthoryear{Wright}{1978}]{wright1978estimating}
\begin{barticle}[author]
\bauthor{\bsnm{Wright},~\bfnm{Farrell~T}\binits{F.~T.}}
(\byear{1978}).
\btitle{Estimating strictly increasing regression functions}.
\bjournal{Journal of the American Statistical Association}
\bvolume{73}
\bpages{636--639}.
\end{barticle}
\endbibitem

\end{thebibliography}






\end{document}